\documentclass[11pt,a4paper]{amsart}
\calclayout
\numberwithin{equation}{section}
\usepackage{extarrows}
\allowdisplaybreaks
\theoremstyle{plain}

\newtheorem{theorem}{Theorem}[section]
\newtheorem{lemma}[theorem]{Lemma}
\newtheorem{prop}[theorem]{Proposition}

\newtheorem{corollary}[theorem]{Corollary}

\numberwithin{equation}{section}
\usepackage{enumerate}

\usepackage{amssymb}
\usepackage{mathtools}
\mathtoolsset{showonlyrefs}
\usepackage{mathrsfs}
\usepackage{comment}
\usepackage{hyperref}
\usepackage{xcolor}
\usepackage{footnote}

\begin{document}
	\title[	]{Asymptotic behaviours of critical branching random walk in $\mathbb{R}^d$}

	\thanks{The research of Haojie Hou is supported by the China Postdoctoral Science Foundation (No. 2024M764112).}
	\thanks{The research of Yaping Zhu is supported  by National Natural Science Foundation of China, Tianyuan Mathematics Young Researcher Project (Grant No. 12526540) and the Fundamental Research Funds for the Central Universities. }
	\thanks{*Yaping Zhu is the corresponding author}
	\author[H. Hou]{Haojie Hou}
	\author[Y. Zhu]{Yaping Zhu$^\ast$}
	\address{Haojie Hou\\ School of Mathematics and Statistics \\ Beijing Institute of Technology \\ Beijing 100081\\ P. R. China}
	\email{houhaojie@bit.edu.cn}
	
	\address{Yaping Zhu\\ 	Department of Mathematics\\ Shanghai University of Finance and Economics\\ Shanghai 200433\\ P. R. China}
	\email{zhuyaping@mail.sufe.edu.cn}

	\begin{abstract}
		In this paper, we study the asymptotic behaviours of a critical branching random walk in $\mathbb{R}^d$ under the assumption that the offspring distribution belongs to the domain of attraction of an $\alpha$-stable law with $\alpha\in(1,2]$, and that the jump distribution has a finite $\frac{2\alpha}{\alpha-1}$-th moment.  
		First, we establish the precise decay rate for the tail probability of the all-time maximal displacement $M^d$. Next, we investigate the maximal displacement $M_n^d$ at generation $n$ and prove a conditional limit theorem for the distribution of $M_n^d$ given that the process survives up to generation $n$. These results extend the corresponding 1-dimensional results of Lalley and Shao \cite{Lalley15} to the case $d\ge2$. 
		Finally, we study the asymptotic behaviour of the total progeny $\zeta$. In particular, we show that, conditioned on the event $\{M^d\ge x\}$,  $\zeta$ converges in distribution under an appropriate normalization. This result reveals a quantitative relationship between the maximal displacement and the total progeny size.
	\end{abstract}
	\maketitle

\section{Introduction}
\subsection{Background introduction}
A discrete-time branching random walk (BRW) in $\mathbb{R}^d$ is defined as follows. Initially, there is a single particle located at the origin. Each particle lives for one unit of time, after which it dies and produces a random number of offspring. If a particle is located at $\mathbf{x}\in\mathbb{R}^d$, then the positions of its offspring in the next generation are formed by the point process
$\mathbf{x}+N\delta_{\mathbf X}$. We assume that all particles evolve independently and according to the same law. The jump distribution 
$\mathbf{X}=(X_1,X_2,\ldots,X_d)^T$ 
is assumed to be a $d$-dimensional random vector
satisfying
\[
\mathbb{E} \mathbf{X}= \mathbf{0}\quad \mbox{and} \quad \eta^2:= \mathbb{E}(\|\mathbf{X}\|^2)\in (0,\infty)
.
\]
As for the offspring distribution, we assume that $N$ is independent of $\mathbf{X}$,
equal in law to
$\{p_k,\; k\in\mathbb{N}_0\}$.
The BRW is called supercritical, critical or subcritical if $\mathbb{E}(N)$ is larger than, equal to or smaller than $1$. 
In this paper, we focus on the critical case, in which the BRW dies out in finite time almost surely.

It is well known that a BRW possesses an underlying random tree structure. Let $\mathbb{T}$ denote the genealogical tree of the branching process rooted by $\emptyset$ with offspring distribution $\{p_k\}_{k\ge 0}$. Let $\mathbf{X}_{\emptyset}=\mathbf{0}$ and $\{\mathbf{X}_v,v\in \mathbb{T} \backslash \emptyset\}$ be i.i.d. copies of $\mathbf{X}$, where $\mathbf{X}_v$ stands for the displacement of particle $v$. Let $V(w):=\sum_{v\prec w}\mathbf{X}_v$ be the position of particle $w\in \mathbb{T}$, where $v\prec w$ means that $v=w$ or $v$ is an ancestor of $w$. For each particle $w\in \mathbb{T}$, we use $|w|$ to denote the generation of $w$.  The population and their locations at generation $n$ forms a point process given by 
\begin{equation}
	Z_n:= \sum_{|w|=n}\delta_{V(w)}.
\end{equation}
For any particle $w$, let $\|V(w)\|$ be the Euclid Norm of $V(w)$. 
For $d=1$, define 
\begin{equation}
	M_n^{1}:=\max_{|w|=n}V(w) \quad \mbox{and}\quad M^{1}:=\sup_{n\ge 0} M_n^{1}.
\end{equation}
For $d\geq 2$, define 
\begin{equation}\label{all-time-maximum}
	M_n^{d}:=\max_{|w|=n}\|V(w)\| \quad \mbox{and}\quad M^{d}:=\sup_{n\ge 0} M_n^{d}.
\end{equation}
Throughout this paper, for a finite point measure $\mu$, we write $\mathbb{P}_{\mu}$ for the probability law of the BRW with initial value $Z_0=\mu$, and write $\mathbb{E}_{\mu}$ for the associated expectation. For simplicity, we write 
$\mathbb{P}:=\mathbb{P}_{\delta_\mathbf 0}$ and $\mathbb{E}:=\mathbb{E}_{\delta_\mathbf 0}$ . 

At present, studies on the long-time behaviours of critical BRW have been mainly focused on the case $d=1$; see for example, Lalley and Shao \cite{Lalley15}. It was proved therein that if $\sum_{k=0}^\infty k^3 p_k<\infty$ and 
the jump distribution $X$ satisfying 
$\mathbb{E}(|X|^{4+\varepsilon})<\infty$ for some $\varepsilon>0$, then
\begin{equation}\label{tail-one-dimensional}
	\lim_{x\to\infty}x^2 \mathbb{P}\big(M^1 \ge x\big) =\frac{6\eta^2}{\sigma^2},
\end{equation}
where $\sigma^2$ is the variance of the offspring distribution. Moreover, it was also proved in Lalley and Shao \cite{Lalley15} that the conditioned law of $M_n^1/\sqrt{n}$ under survival event $Z_n(\mathbb{R})>0$ converges weakly. Recently, Leh\'ericy \cite{Lehericy} and Zhang \cite{ZX25} weakened the condition of \cite{Lalley15}, showing that the condition on the jump distribution can be weakened to $\mathbb{P}(X>x)= o(x^{-4})$, while the offspring distribution only requires  a finite second moment.  It is also worth mentioning that Leh\'ericy \cite[Theorem 2.5]{Lehericy}, studied the convergence in distribution of the total progeny of the 1-dimensional BRW, conditioned on the right-most particle attaining a large position.

Motivated by the above results, we further investigate the long-time asymptotic behaviours of critical BRW for $d\geq 2$. In particular, we study the decay rate of $\mathbb{P}(M^d\geq x)$ as $x\to\infty$, the convergence of $M_n^d$ under a suitable normalization, and the asymptotic behavior of the total progeny $\zeta$. We work under a more general framework.  In what follows, we always assume that 

\begin{itemize}
	\item [{\bf(H1)}]  The offspring distribution $\{p_k\}_{k\geq 0}$ satisfies $\sum_{k=1}^{\infty}kp_k=1, p_1<1$ and belongs to the domain of attraction of an $\alpha$-stable distribution with $\alpha\in (1,2]$.
	More precisely, either there exist $\alpha\in (1,2)$ and $\kappa(\alpha)\in(0,\infty)$ such that
	\[
	\lim_{n\to\infty} n^\alpha \sum_{k=n}^\infty p_k = \kappa(\alpha),
	\]
	or that (corresponding $\alpha=2$) $\sum_{k=0}^\infty k^2 p_k<\infty.$

	\item[{\bf(H2)}] 
	The jump distribution $\mathbf{X}$ is independent of $N$ and satisfies $\mathbb{E}(\mathbf{X})=\mathbf{0}, \eta^2= \mathbb{E}(\|\mathbf{X}\|^2)\in (0,\infty)$ and 
	$\mathbb{E}\big(\|\mathbf{X}\|^{\frac{2\alpha}{\alpha-1}}\big)<\infty$.
\end{itemize}

In $d=1$, under  {\bf(H1)} and  a slightly weaker moment condition than {\bf(H2)}, Zhang \cite{ZX25} proved that 
\begin{equation}\label{tail-one-dimensional-alpha-stable}
	\lim_{x\to\infty} x^{\frac{2}{\alpha-1}}\mathbb{P}(M^1\ge x) 
	=: 
	\Gamma_\alpha \in (0,\infty).
\end{equation}
A first glance to the order of $\mathbb{P}(M^d\ge x)$ in high dimension $d\geq 2$ can be obtained from \eqref{tail-one-dimensional-alpha-stable}. 
Wirting  $V(w)=\big(V^{(1)}(w), V^{(2)}(w),\cdots,V^{(d)}(w)\big)^T$ and  define 
\begin{equation}
	M^{d, (j),+}:=\sup_{n\ge 0}\max_{|w|=n}V^{(j)}(w)\quad \mbox{and} \quad M^{d, (j),-}:=-\inf_{n\ge 0}\min_{|w|=n}V^{(j)}(w).
\end{equation}
Combining the following inequality 
$$\sup_{1\leq j\leq d} M^{d, (j),+},  \sup_{1\leq j\leq d} M^{d, (j),-}\le  M^d \leq \sum_{j=1}^{d}M^{d, (j),+} + \sum_{j=1}^{d}M^{d, (j),-}$$
and
\eqref{tail-one-dimensional-alpha-stable},  we deduce that under   {\bf(H1)} and {\bf(H2)},
\begin{equation}\label{rough-behavior-M}
	0< \liminf_{x\to\infty} x^{\frac{2}{\alpha-1}} \mathbb{P}(M^d\ge x) \leq \limsup_{x\to\infty} x^{\frac{2}{\alpha-1}} \mathbb{P}(M^d\ge x)<\infty.
\end{equation}
This shows that the leading-order decay rate of $\mathbb{P}(M^d\ge x)$ is  $x^{-\frac{2}{\alpha-1}}$ as $x\to\infty$.

At the end of this subsection, we briefly review some related results on critical branching L\'evy process, which serves as the continuous-time analogue of the BRW.
We begin with the Brownian case. Let \(N_t\) denote the set of particles alive at time \(t\), and for each \(u\in N_t\), let \(\mathbf{X}_u(t)\) denote the position of particle \(u\) at time \(t\). 
For any fixed \(a>0\), Sawyer and Fleischman \cite{SF79} derived the tail asymptotics for the following function
\begin{equation}
	u(\mathbf{x}):=
	\left\{
	\begin{array}{ll}
		\mathbb{P}_{\mathbf{0}}(M^1\ge x), & \mbox{when } d=1,\\[1mm]
		\mathbb{P}_{\mathbf{x}}\!\left(\inf_{t>0}\inf_{u\in N_t}\|\mathbf{X}_u(t)\|<a\right), & \mbox{when } d\ge 2
	\end{array}
	\right.
\end{equation}
as \(x\to\infty\) for \(d=1\) and as \(\|\mathbf{x}\|\to\infty\) for \(d\ge 2\) based on the key observation that the function $u$ is the unique solution to the equation
\begin{equation}\label{SF-PDE}
	\frac{1}{2}\Delta u=-\beta\Big(\sum_{k=0}^\infty p_k u^k-u\Big),
\end{equation}
where $\beta$ is the branching rate and $\{p_k\}_{k\geq 0}$ is the offspring distribution.

Over the past several years, there have been many extensions of the work in Sawyer and Fleischman \cite{SF79} for the case \(d=1\).  
When the
Brownian motion is replaced by a centered L\'evy process, Hou et al. \cite{HJRS2025} proved that, if \(\lim_{n\to\infty}n^\alpha\sum_{k=n}^\infty p_k\in(0,\infty)\) exists and the spatial motion has a light positive tail, then an analogue of \cite{SF79} still holds, with the polynomial order \(x^2\) replaced by \(x^{\frac{2}{\alpha-1}}\). For non-centered L\'evy processes, we refer to Profeta \cite{Profeta2024} for the spectrally negative case and to Profeta \cite{Profeta2025} for the general case. Moreover, if the L\'evy process does not admit a second moment, then the tail behaviour of \(M^1\) is also different from that in \cite{SF79}; see \cite{HJRS2026, Lalley16, Profeta2022}.

\subsection{Critical super-Brownian motion}
In this subsection, we present several results on super-Brownian motion that will be used in the proofs of the main results.
For any domain $D \subset \mathbb{R}^d$, 
let $C_b^+(D)$ be the space of non-negative continuous bounded funciton on $D$ and let $B_b^+(D)$ be the space of non-negative  bounded funciton on $D$.
In the remainder of the paper, we write $X=\{X_t\}_{t\geq 0}$ for the 
$\mathbb{R}^d$-valued super-Brownian motion which appears as
the scaling limit of the critical BRW $\{Z_n\}_{n\ge 0}$.
More precisely, the spatial motion of $X$ is a $d$-dimensional Brownian motion $(W_t, \mathbb{P}_\mathbf{y})$ with  covariance
matrix $\Sigma=\mathbb{E}(\mathbf{X}\mathbf{X}^{\mathsf T})$ and 
the branching mechanism $\psi$ of $X$ is given by
\begin{equation}\label{Stable-branching}
	\psi(u):=	\mathcal{C}(\alpha) u^\alpha :=	\left\{\begin{array}{ll}
		\frac{	\kappa(\alpha)\Gamma(2-\alpha)}{\alpha-1} u^\alpha,\quad &\mbox{when}\  \alpha\in (1,2),\\
		\frac{1}{2}
		\left(\sum_{k=1}^\infty k(k-1)p_k\right)u^2 ,\quad &\mbox{when}\ \alpha=2.
	\end{array}\right.
\end{equation}
Equivalently, for every $\varphi\in C_b^+(\mathbb{R}^d)$, we have the 
log-Laplace representation
\begin{equation}
	\mathbb{E}_{\mu}\!\left[\exp\!\left(-\langle \varphi,  X_t\rangle\right)\right]
	=\exp\!\left(-\langle  u_{\varphi}(t,\cdot), \mu\rangle\right),
\end{equation}
where $u_{\varphi}(t,\mathbf{x})$ is the unique non-negative solution to the integral equation
\begin{equation}\label{pde}
	u_\varphi(t,\mathbf{x})
	= \mathbb{E}_{\mathbf{x}}(\varphi(W_t))
	-\mathcal{C}(\alpha) \int_0^t \mathbb{E}_{\mathbf{x}}\big( u_{\varphi}^\alpha(t-s, W_s) \big)\mathrm{d}s
	,\quad t\geq 0, \mathbf{x}\in \mathbb{R}^d.
\end{equation}
It is well-known (for example, see Kyprianou \cite[Theorem 12.5]{Kybook})
that the super-Brownian motion with branching mechanism $ \psi$ dies out at finite time almost surely.

Now we introduce the excursion measures $\{\mathbb{N}_\mathbf{y}, \mathbf{y}\in \mathbb{R}^d\}$ associated with $X$. 
Denote by $\mathcal{M}_F(\mathbb{R}^d)$ the space of finite Borel measures on $\mathbb{R}^d$. Let $\mathcal{W}_0^+$ denote the space of right continuous paths from $[0,+\infty)$ to $\mathcal{M}_F(\mathbb{R}^d)$ having zero as a trap.  
It follows from Li \cite[Corollary 2.8 and Theorem A.41]{LZbook} 
that there exists a unique family of $\sigma$-finite measures $\{\mathbb{N}_\mathbf{y}:\mathbf{y}\in\mathbb{R}^d\}$ on $\mathcal{W}_0^+$ such that 
for any $t>0$, $\mathbf{y}\in\mathbb{R}^d$ and $f \in B_b^+(\mathbb{R}^d)$,
\begin{equation}\label{N-measures}
	\mathbb{N}_\mathbf{y} \big(1-e^{-\left\langle f, w_t\right\rangle}\big)
	=-\log \mathbb{E}_{\delta_\mathbf{y}} \big(e^{-\langle f, X_t\rangle}\big).
\end{equation}
$\{\mathbb{N}_\mathbf{y} :\mathbf{y} \in\mathbb{R}^d\}$ are called the $\mathbb{N}$-measures associated to $\{\mathbb{P}_{\delta_\mathbf{y}}:\mathbf{y}\in\mathbb{R}^d\}$.

Next, we introduce the exit measure related to $X$. For any Borel measurable set $D\subset \mathbb{R}^d$, denote by
\[
\tau_D := \inf\{t>0 : W_t \notin D\}
\]
the first exit time of the Brownian motion $(W_t)_{t\ge0}$ from the domain $D$. According to Dynkin \cite[Sections 1.2 and 1.3]{Dynkin91}, each Borel measurable set $D$ corresponds to a pair of random measures $X^D$ and $Y^D$ such that 
for any 
$g \in B_b^+(\partial D)$ and  $f \in B_b^+(D)$,
the function
\begin{equation}\label{eq:Dynkin-Laplace-functional}
	V_{f,g}^D(\mathbf y):= - \log \mathbb E_{\delta_{\mathbf y}}
	\left[	\exp\left\{	- \langle g, X^D \rangle	-  \langle f, Y^D \rangle 	\right\}	\right]
\end{equation}
is the unique solution to the following integral equation
\begin{equation}\label{eq:Dynkin-integral-equation}
	V(\mathbf y)
	= \mathbb E_{\mathbf y}\!\left[g\!\left(W_{\tau_D}\right)\right]
	+ \mathbb E_{\mathbf y}\!\left[\int_0^{\tau_D} f(W_s)\,\mathrm ds\right]
	- \mathbb E_{\mathbf y}\!\left[\int_0^{\tau_D} \psi\!\left(V(W_s)\right)\,\mathrm ds\right].
\end{equation}
In the literature, usually $X^D$ is refered to as the exit measure whose support is a subset of $\partial D$.  

For $\mathbf{x}\in \mathbb{R}^d$ and 
$r\geq 0$,
let $D(\mathbf{x},r):=\{ \mathbf{z}\in \mathbb{R}^d:\| \mathbf{x}-\mathbf{z}\|< r\}$ and $\overline{D}(\mathbf{x},r):=\{ \mathbf{z}\in \mathbb{R}^d:\| \mathbf{x}-\mathbf{z}\|\leq r\}$ denote the open and closed balls centered at the $\mathbf{x}$ with radius $r$.
For notational convenience, we write $D_r:= D(\mathbf{0},r)$ and $\overline{D}_r:= \overline{D}(\mathbf{0},r)$.
In our paper, we focus on the special case $D=\overline{D}_1$.  The formal definition of $Y^{\overline{D}_1}$ was given in Dynkin \cite[Theorem 1.4]{Dynkin1991-2}, but we are going to  adapt the explanation in Dynkin and Kuznetsov \cite[Section 0.2]{DK1997}. Let $X_t^{D_1}$ be the restriction of $X_t$ on $D_1$ that denotes the mass distribution of particles still in $D_1$ up to time $t$, 
then 
\begin{align}\label{expression-of-Y}
	Y^{\overline{D}_1}= \int_0^\infty X_t^{D_1}\mathrm{d}t.
\end{align}

\subsection{Main results}
In this subsection, we present the main results of this paper. 
Define
\begin{equation}\label{Def-of-maximal}
	M_t^{X}:=\inf\{r\ge 0: w_t\left(D_r^c\right)=0\}\quad \mbox{and}\quad M^{X,d}:= \sup_{t>0} M_t^{X}. 
\end{equation}
The following theorem establishes the precise decay rate of the tail probability of $M^d$.
\begin{theorem}\label{thm-behavior-all-time-maximum}
	
	Assume {\bf(H1)} and {\bf(H2)}.
	For any $d \ge 1$ and $\mathbf{y} \in D_1$, we have 
	\begin{equation}
		\lim_{x\to\infty}x^{\frac{2}{\alpha-1}}\mathbb{P}_{\delta_{x\mathbf{y}}}\big(M^d\geq x\big)=\lim_{x\to\infty}x^{\frac{2}{\alpha-1}}\mathbb{P}_{\delta_{x\mathbf{y}}}\big(M^d> x\big)= \mathbb{N}_\mathbf{y}\big(M^{X,d}\geq 1\big).
	\end{equation}
\end{theorem}

The following theorem, the second main result of this paper, describes the asymptotic 
behaviour
of $M_n^d$ as $n\to\infty$ and extends Lalley and Shao \cite[Theorem 3]{Lalley15}, to dimensions $d\geq 2$.

\begin{theorem}\label{thm-behavior-Mn}
	
	Assume {\bf(H1)} and {\bf(H2)}.
	For any $\mathbf{y}\in \mathbb{R}^d$ and 
	$r> 0$, we have 
	\begin{equation}
		\lim_{n\to\infty}n^{\frac{1}{\alpha-1}}\mathbb{P}_{\delta_{\sqrt{n}\mathbf{y}}}\big(M_n^d > \sqrt{n}r\big)
		=\mathbb{N}_\mathbf{y}\big( M_1^{X}> r\big).
	\end{equation}
\end{theorem}
Consequently, we have the following Yaglom-type conditional limit theorem.
\begin{corollary}\label{Cor}
	
	Assume {\bf(H1)} and {\bf(H2)}.
	As $n\to \infty$, the following convergence in distribution holds:
	\begin{equation}
		\mathbb{P}_{\delta_{\sqrt{n}\mathbf{y}}}\Big(\frac{M_n^d}{\sqrt{n}} \in \cdot \Big| Z_n(\mathbb{R}^d)>0\Big)  
		\xRightarrow[]{\ d\ } \mathbb{N}_{\mathbf{y}}\big(M_1^X \in \cdot|w_1(\mathbb{R}^d)\neq 0\big).
	\end{equation}
\end{corollary}
Define
\begin{equation}
	\zeta^X:=\int_0^\infty \langle 1, w_t\rangle\mathrm{d}t.
\end{equation}
The asymptotic behaviour for the total progeny $\zeta$ of the BRW is proved in the following theorem.

\begin{theorem}\label{thm-total-progeny}
	
	Assume {\bf(H1)} and {\bf(H2)}. 
	For each $  \mathbf{y} \in D_1$, as $x\to\infty$,  
	\begin{equation}
		\mathbb{P}_{\delta_{x\mathbf{y}}}\Big(\frac{\zeta}{x^{\frac{2\alpha}{\alpha-1}}}\in \cdot \Big|M^d\geq x\Big)\xRightarrow[]{\ d\ }  \mathbb{N}_\mathbf{y}\big(\zeta^X\in \cdot \big| M^{X,d}\geq 1\big).
	\end{equation}
\end{theorem}

\textbf{Organization of the paper.} 
The rest of the paper is organized as follows. In Section \ref{S2}, we present and prove several elementary properties of high-dimensional random walks that will be used in the proofs of the main results. In Section \ref{S3}, we introduce a discrete-time version of the exit measure 
$Z_{\overline{D}_x}$ 
(see \eqref{def-exit-measure}) and a modified random variable $\zeta_x$ (see \eqref{def-modified-total-progeny}) associated with $\zeta$. We then derive an integral equation for the Laplace transform of $(\zeta_x, \langle 1, Z_{\overline{D}_x} \rangle)$. The convergence of the solution to this integral equation is established in Section \ref{S4.1}, while the proofs of the main results are presented in Sections \ref{S4.2}--\ref{S4.4}.

\section{Basic properties for high dimensional centered random walk}\label{S2}

Let $\mathbf{S}_n$  be an $\mathbb{R}^d$-valued random walk under $\mathbb{P}_{\mathbf{y}}$  
such that $\mathbf{S}_0:= \mathbf{y}$ and $(\mathbf{S}_n- \mathbf{S}_{n-1}, n\geq 1)$ are i.i.d. copies of  $\mathbf{X}$. 
Define 
\begin{equation}
	\tau_x:= \inf\big\{n\geq 0: \mathbf{S}_n \notin \overline{D}_x \big\} =  \inf\left\{n\geq 0: \| \mathbf{S}_n\|>x  \right\} .
\end{equation}
According to the independence of $\mathbf{S}_n- \mathbf{S}_{n-1}$, it is easy to see that for any $\mathbf{y}\in \mathbb{R}^d$,
\begin{equation}\label{lemma-martingale}
	\left\{ \| \mathbf{S}_n \|^2- \eta^2 n, n\geq 1 \right\}\quad \mbox{is a } \mathbb{P}_{\mathbf{y}}-\mbox{martingale}, 
\end{equation}
where $\eta^2= \mathbb{E}(\| \mathbf{X}\|^2).$ For simplicity, we write $\mathbb{P}:=\mathbb{P}_{\mathbf{0}}$ and $\mathbb{E}:=\mathbb{E}_{\mathbf{0}}$. Taking $\{x_k\}_{k\geq 0}\subset [0,\infty)$ satisfying $\lim_{k\to\infty}x_k=\infty$, define the scaled process
\begin{align}
	\mathbf{S}_t^{(k)}:=x_k^{-1}\mathbf{S}_{\lfloor t x_k^2\rfloor}, \quad t\ge 0.
\end{align}
By Donsker's invariance principle, under $\mathbb{P}_{x_k\mathbf{y}}$, the process $\mathbf{S}^{(k)}$ converges in distribution in the Skorokhod space  $D\left([0,\infty),\mathbb{R}^d\right)$ to a Brownian motion $W=\{W_t\}_{t\ge 0}$ starting from $\mathbf{y}$ with covariance matrix $\Sigma$. Moreover, as $k\to \infty$, the following  joint convergence holds 
in distribution on the product space $D\left([0,\infty),\mathbb{R}^d\right) \times \mathbb{R}_+$: 
\begin{align}\label{joint-convergence}
	\left(\mathbf{S}^{(k)}, x_k^{-2}\tau_{x_k}\right) \xLongrightarrow{k\to\infty} \left(W,\tau_{1}^W\right)
\end{align}
where $\tau_{1}^W:=\inf\{t>0: W_t \notin \overline{D}_1\}$.
\begin{lemma}\label{lemma-step1}
	There exists $\kappa_0>0$ such that 
	\begin{equation}
		\sup_{x>1} \sup_{\mathbf{y}\in \overline{D}_1} \mathbb{E}_{x\mathbf{y}}\Big(e^{\kappa_0 x^{-2}\tau_x}\Big)<\infty.
	\end{equation}
	Consequently, we have
	\begin{equation}\label{step1}
		C_\tau:= \sup_{x>1}\sup_{\mathbf{y}\in \overline{D}_1}
		x^{-2}
		\mathbb{E}_{x\mathbf{y}}\left( \tau_{x}\right)<\infty\quad \mbox{and}\quad \lim_{M\to\infty} \sup_{x> 1}\sup_{\mathbf{y}\in \overline{D}_1}
		x^{-2}
		\mathbb{E}_{x \mathbf{y}} \big( \tau_{x} 1_{\{\tau_{x} > M x^2\}} \big)=0. 
	\end{equation}	
\end{lemma}
\begin{proof}
	According to the central limit theorem, there exists a large $n_0$ such that 
	$$\sup_{x> 1} \mathbb{P}\big(\|\mathbf{S}_{\lfloor n_0 x^2\rfloor}\|\leq 2 x\big)\leq 2^{-1}.$$ Therefore, for each $p\in\mathbb{N}$ and $\mathbf{y}\in \overline{D}_1$, we have 
	\begin{equation}
		\begin{aligned}
			& \mathbb{P}_{x \mathbf{y}} \big(\tau_{x} > p \lfloor n_0 x^2 \rfloor\big) \leq 	\mathbb{P}_{x \mathbf{y}} \Big( \max_{1\leq j\leq p} \|\mathbf{S}_{j \lfloor n_0 x^2 \rfloor }\| \leq x \Big)  \\
			& \leq  \mathbb{P}_{x\mathbf{y}} \Big( \max_{1\leq j\leq p} \|\mathbf{S}_{j \lfloor n_0 x^2 \rfloor } - \mathbf{S}_{(j-1) \lfloor n_0 x^2 \rfloor } \| \leq 2x \Big) 
			=  \mathbb{P}\big( \| \mathbf{S}_{ \lfloor n_0 x^2 \rfloor } \| \leq 2x \big)^p \leq 2^{-p}. 
		\end{aligned}
	\end{equation}
	Now for any $z>0$, there exists $p\in \mathbb{N}$ such that $p \lfloor n_0 x^2 \rfloor  \leq z x^2 < (p+1)\lfloor n_0 x^2 \rfloor \leq (p+1)n_0 x^2$ so that 
	$$\sup_{x>1}	\sup_{\mathbf{y} \in \overline{D}_1} \mathbb{P}_{x\mathbf{y}}\big(\tau_{x} > z x^2 \big)\leq 2^{-p}\leq 2^{1- \frac{z}{n_0}}.
	$$
	Noticing that $n_0$ is independent of $x$ and $\mathbf{y}$, for any $\kappa_0 < \frac{1}{n_0} \log 2$,
	\begin{equation}
		\begin{aligned}
			\sup_{x>1}	\sup_{\mathbf{y} \in \overline{D}_1} \mathbb{E}_{x\mathbf{y}}\Big(e^{\kappa_0 x^{-2}\tau_x}\Big) & = 1+  \kappa_0 \sup_{x>1}	\sup_{\mathbf{y} \in \overline{D}_1} \int_0^\infty e^{\kappa_0 z} \mathbb{P}_{x\mathbf{y}}\big(\tau_{x} > z x^2 \big)  \mathrm{d}z\\
			&\leq 1+  \kappa_0  \int_0^\infty e^{\kappa_0 z} 2^{1- \frac{z}{n_0}} \mathrm{d}z <\infty,
		\end{aligned}
	\end{equation}
	which implies the first result. The second result is a direct consequence of the first result and we omit the details here. 
	
\end{proof}

\begin{lemma}\label{Tightness} 
	For any $A>0$, it holds that 
	\begin{equation}
		\begin{aligned}
			& \sup_{x>1} \sup_{\mathbf{y}\in \overline{D}_1} \frac{1}{x^2}\mathbb{P}_{x\mathbf{y}}\left(\|\mathbf{S}_{\tau_{x}}\|>x+A\right)
			\le 4C_\tau \mathbb{P}( \|\mathbf{X}\|>A ). 
		\end{aligned}
	\end{equation}
\end{lemma}
\begin{proof}
	Since $\tau_x\geq 1$ holds $\mathbb{P}_{x\mathbf{y}}$-almost surely for all $\mathbf{y}\in \overline{D}_1$, we have
	\begin{align}\label{upper-bound-tail-probab}
		&\mathbb{P}_{x\mathbf{y}}\left(\|\mathbf{S}_{\tau_{x}}\|>x+A\right)
		=\sum_{k=1}^{\infty}\mathbb{P}_{x\mathbf{y}} \left(\tau_{x}=k,\|\mathbf{S}_{k}\|>x+A\right) \nonumber \\
		&=\sum_{k=1}^{\infty}\mathbb{P}_{x\mathbf{y}} \Big(\max_{0\leq j\leq k-1} \|\mathbf{S}_{j}\|\leq x,  \|\mathbf{S}_{k}\|>x+A\Big) \nonumber\\
		&\le \sum_{k=1}^{\infty}\mathbb{P}\Big( \max_{0\leq j\leq k-1} \|\mathbf{S}_{j}\|\leq x(1+\|\mathbf{y}\|),  \| \mathbf{S}_k - \mathbf{S}_{k-1}\| >A\Big)\nonumber\\
		& =\mathbb{P}( \|\mathbf{X}\|>A)\sum_{k=1}^{\infty}\mathbb{P}\left(\tau_{x(1+\|\mathbf{y}\|)} \ge k\right)
		=\mathbb{P}( \|\mathbf{X}\|>A )\mathbb{E}(\tau_{x(1+\|\mathbf{y}\|)}).
	\end{align}	
	Combining Lemma \ref{lemma-step1}, \eqref{upper-bound-tail-probab} and the fact that $\mathbf{y}\in \overline{D}_1$, we conclude that for any $x>1$, $A>0$ and $\mathbf{y}\in \overline{D}_1$, 
	\begin{equation}
		\mathbb{P}_{x\mathbf{y}}\left(\|\mathbf{S}_{\tau_{x}}\|>x+A\right)\leq C_\tau x^2 (1+\|\mathbf{y}\|)^2 \mathbb{P}( \|\mathbf{X}\|>A )\leq 4C_\tau x^2  \mathbb{P}( \|\mathbf{X}\|>A ),
	\end{equation}
	which implies the desired result. 
\end{proof}

For any $\mathbf{z}  \in \overline{D}_1$  and $x>1$, define 
\begin{equation}\label{Def-tau-x-z}
	\tau_{x}^{\mathbf{z}}:=\inf\{n\geq 0 : \mathbf{S}_n+x\mathbf{z} \notin \overline{D}_x\}.
\end{equation}

\begin{lemma}\label{lemma-step2}
	For each $\delta\in (0,1)$, it holds that 
	\begin{equation}
		\limsup_{x\to +\infty}\sup_{\mathbf{z} \in \overline{D}_\delta}\sup_{\mathbf{y},\mathbf{y}+\mathbf{z}\in  \overline{D}_1}\frac{1}{x^2}  \mathbb{E}_{x\mathbf{y}}\left(|\tau_x-\tau_{x}^\mathbf{z} |\right) \leq  \frac{9\delta}{\eta^2}.
	\end{equation}
\end{lemma}
\begin{proof}
	Combining \eqref{lemma-martingale} and the optional stopping theorem, $\{\| \mathbf{S}_{n\land \tau_{x}}\|^2-\eta^2 (n\land \tau_{x})\}_{n\ge 1}$ is also a 
	$\mathbb{P}_{x\mathbf{y}}$-martingale, 
	which implies that 
	\begin{align}\label{e1}
		&\eta^2  \mathbb{E}_{x\mathbf{y}} \left(n\land \tau_{x}\right)= 	\mathbb{E}_{x\mathbf{y}}\big(\|\mathbf{S}_{n\land \tau_{x}}\|^2\big) 
		-x^2 \|\mathbf{y}\|^2 \nonumber \\
		=& \mathbb{E}_{x\mathbf{y}}\big(\|\mathbf{S}_{\tau_{x}}\|^2 
		1_{\{\tau_{x}\leq n\}}\big) 
		+\mathbb{E}_{x\mathbf{y}}\big(\|\mathbf{S}_{n}\|^2 
		1_{\{\tau_{x} > n\}}\big)
		-x^2 \|\mathbf{y}\|^2.
	\end{align}
	Noticing that the second term is bounded by 
	$x^2 \mathbb{P}_{x\mathbf{y}}(\tau_{x} > n)$,  which tends to $0$ as $n\to \infty$.
	Therefore, letting $n\to\infty$ in \eqref{e1}, we deduce that 
	\begin{equation}\label{important-eq}
		\mathbb{E}_{x\mathbf{y}}\big(\|\mathbf{S}_{\tau_{x}}\|^2 \big)
		-x^2 \|\mathbf{y}\|^2
		=\eta^2 \mathbb{E}_{x\mathbf{y}} \left(\tau_{x}\right).
	\end{equation}
	Since $\{\|\mathbf{S}_n+x\mathbf{z}\|>x\}\subset \{\|\mathbf{S}_n\|>x(1-\|\mathbf{z}\|)\}$ and  $\{\|\mathbf{S}_n\|>x(1+\|\mathbf{z} \|)\} \subset \{\|\mathbf{S}_n+x\mathbf{z}\|>x\}$,  we obatin
	\begin{equation}
		\tau_{x(1-\|\mathbf{z}\|)} \le \tau_{x}^\mathbf{z}\quad \text{and}  \quad \tau_x  \le \tau_{x(1+\|\mathbf{z}\|)}.
	\end{equation}
	Therefore, it follows from  \eqref{important-eq} that 
	\begin{align}\label{e3}
		& \eta^2 \mathbb{E}_{x\mathbf{y}}\left(|\tau_x-\tau_{x}^\mathbf{z}|\right)
		\le \eta^2 \mathbb{E}_{x\mathbf{y}}\left(\big|\tau_{x(1+\|\mathbf{z}\|)}- \tau_{x(1-\|\mathbf{z}\|)}\big|\right)\nonumber \\
		= 
		&\ \mathbb{E}_{x\mathbf{y}}\big(\|\mathbf{S}_{\tau_{x(1+\|\mathbf{z} \|)}}\|^2-\|\mathbf{S}_{\tau_{x(1-\|\mathbf{z}\|)}}\|^2\big) 
		\le 
		\mathbb{E}_{x\mathbf{y}}\big(\|\mathbf{S}_{\tau_{x(1+\|\mathbf{z}\|)}}\|^2\big) -x^2(1-\|\mathbf{z}\|)^2. 
	\end{align}
	According to Lemma \ref{Tightness} (with $x$ replaced by $x(1+\| \mathbf{z}\|)$ and $\mathbf{y}$ replaced by $(1+\| \mathbf{z}\|)^{-1}\mathbf{y}$) , we see that for any $A>0$, 
	\begin{equation}
		\begin{aligned}  
			\mathbb{P}_{x\mathbf{y}}\big(
			\|\mathbf{S}_{\tau_{x(1+\|\mathbf{z}\|)}}\|
			>x(1+\|\mathbf{z} \|)+A\big)
			\le 4C_\tau x^2 (1+ \|\mathbf{z}\|)^2 \mathbb{P}\left(\|\mathbf{X}\|>A\right)
			\leq 16C_\tau x^2 \mathbb{P}\left(\|\mathbf{X}\|>A\right). 
		\end{aligned}
	\end{equation}
	Therefore, it follows from the above inequality that
	\begin{equation}
		\begin{aligned}    
			&\mathbb{E}_{x\mathbf{y}}\big(\|
			\mathbf{S}_{\tau_{x(1+\|\mathbf{z}\|)}}
			\|^2\big)=2\int_{0}^{\infty}u \mathbb{P}_{x\mathbf{y}}\big(\|
			\mathbf{S}_{\tau_{x(1+\|\mathbf{z}\|)}}
			\|>u\big)\mathrm{d}u\\
			& \le 2\int_{0}^{ x(1+\|\mathbf{z} \|+\delta )}u\mathrm{d}u+ 2\int_{x\delta}^{\infty} (A+ x(1+\|\mathbf{z}\|)) \mathbb{P}_{x\mathbf{y}}\big(\|
			\mathbf{S}_{\tau_{x(1+\|\mathbf{z}\|)}}
			\|>x(1+\|\mathbf{z}\|)+ A\big)\mathrm{d}A \\
			& \leq  x^2(1+\|\mathbf{z}\|+\delta)^2 +32 C_\tau x^2 \int_{x\delta}^{\infty}\left(A+x(1+\|\mathbf{z}\|)\right) \mathbb{P}\left(\|\mathbf{X}\|>A\right)\mathrm{d}A. 
		\end{aligned}
	\end{equation}
	Moreover, according to the inequality
	$A+ x(1+\|\mathbf{z}\|) \le A+ x(1+\delta) \le  \left(1+\frac{1+\delta}{\delta}\right)A$ for all $A \geq x \delta$, 
	we deduce that 
	\begin{equation}\label{e4}
		\begin{aligned}
			\mathbb{E}_{x\mathbf{y}}\big(\|
			\mathbf{S}_{\tau_{x(1+\|\mathbf{z}\|)}}
			\|^2\big) 
			\le x^2(1+\|\mathbf{z}\|+\delta)^2 +32 C_\tau \Big(1+\frac{1+\delta}{\delta}\Big) x^2  \int_{x\delta}^{\infty}A \mathbb{P}\left(\|\mathbf{X}\|>A\right)\mathrm{d}A. 
		\end{aligned}
	\end{equation}
	Combining \eqref{e3} and \eqref{e4}, we conclude that 
	\begin{equation}
		\begin{aligned}
			& \limsup_{x\to\infty} \sup_{ \mathbf{z} \in \overline{D}_\delta} \sup_{\mathbf{y}, \mathbf{y}+\mathbf{z}\in \overline{D}_1} \frac{\eta^2}{x^2}  \mathbb{E}_{x\mathbf{y}}\left(|\tau_x-\tau_{x}^\mathbf{z}|\right) 
			\leq \sup_{ \mathbf{z} \in \overline{D}_\delta} \big((1+\|\mathbf{z}\|+\delta)^2-(1-\|\mathbf{z}\|)^2 \big) \leq 9\delta,
		\end{aligned}
	\end{equation}
	as desired. 
\end{proof}

\section{An integral equation for branching particle systems}\label{S3}
Let $$\mathcal{T}_{\overline{D}_x}:=\big\{u\in \mathbb{T}: V(u)\in \overline{D}_x^c; V(u_j)\in \overline{D}_x,~\forall~1\le j<|u|\big\}$$
be the set of particles in the tree $\mathbb{T}$ that exit $\overline{D}_x$ for the first time.
For any $x >0$, denote by 
\begin{align}\label{def-exit-measure}
	Z_{\overline{D}_x}:=\sum_{w\in \mathcal{T}_{\overline{D}_x}}\delta_{V(w)}
\end{align}
the exit measure, which describes the positions at which the branching random walk first exits $\overline{D}_x$.
Note that under the condition $\sum_{k=1}^{\infty}kp_k=1$, the process $\{Z_n\}_{n\ge 0}$ becomes extinct almost surely in finite time, and hence $Z_{\overline{D}_x}$ is well defined. 
Next, we focus on the asymptotic behaviour of $\zeta$.
Define 
\begin{align}\label{def-modified-total-progeny}
	\zeta_x:=\#\{u \in 
	\mathbb{T}\setminus  \mathcal{T}_{\overline{D}_x}: \forall 	v\prec u, v\notin \mathcal{T}_{\overline{D}_x}   \}.
\end{align}
For any $x>0$ and $\gamma, \theta \geq 0$, define for all $\mathbf{y}\in \overline{D}_x$,
\begin{align}\label{def-v}
	u_{(\gamma, \theta)}(\mathbf{y};x):=1-\mathbb{E}_{\delta_{\mathbf{y}}}\Big(e^{-\gamma \zeta_x-\theta \langle 1, Z_{\overline{D}_x}\rangle} \Big)\quad \text{and} \quad v_{(\gamma, \theta)}(\mathbf{y};x):= \mathbb{E}_{\delta_{\mathbf{y}}}\Big(e^{-\gamma \zeta_x-\theta \langle 1, Z_{\overline{D}_x}\rangle} \Big),
\end{align}
then $	u_{(\gamma, \theta)}(\mathbf{y};x)= 1-	v_{(\gamma, \theta)}(\mathbf{y};x).$ Define 
\begin{align}\label{def-phi}
	\phi(u):=\sum_{k=1}^{\infty}p_k (1-u)^k-(1-p_0)+u,\quad u\in [0,1].
\end{align}

\begin{lemma}\label{expression-small-u}
	For any $ \gamma, \theta \ge 0,\ x>0 $ and $\mathbf{y}\in \overline{D}_x$, 
	\begin{align}
		u_{(\gamma, \theta)}(\mathbf{y};x)
		&=1-\mathbb{E}_{\mathbf{y}}\left(e^{-\gamma \tau_x}\right)
		+\sum_{k=1}^{\infty}p_k (1-e^{-\theta k}) \mathbb{E}_{\mathbf{y}}\left(e^{-\gamma \tau_x}\right)\\
		&\quad  -\sum_{i=1}^{\infty}e^{-\gamma i}\mathbb{E}_{\mathbf{y}}\left(1_{\{\tau_x>i\}}\phi \left(u_{(\gamma,\theta)}(\mathbf{S}_i;x)\right)\right).
	\end{align}
\end{lemma}
\begin{proof}
	Define  $\zeta_x^w:=\#\{u \in \mathbb{T}:  u\notin \mathcal{T}_{\overline{D}_x}, 
	\forall v\prec u,  v\notin \mathcal{T}_{\overline{D}_x}, 
	w\prec u\}.$
	By the branching property, we have $\zeta_x=1+\sum_{|w|=1, w\notin \mathcal{T}_{\overline{D}_x}}\zeta_x^w.$ 
	Therefore, by comparing the norm of the first jump with $x$, we get 
	\begin{align}
		&v_{(\gamma, \theta)}(\mathbf{y};x)
		=e^{-\gamma}\mathbb{E}_{\delta_{\mathbf{y}}}\Big(e^{-\gamma \sum_{|w|=1, w\notin \mathcal{T}_{\overline{D}_x}}\zeta_x^w-\theta \langle 1, Z_{\overline{D}_x}\rangle} \Big)\\
		&=e^{-\gamma}\bigg(p_0+
		\mathbb{P}_{\mathbf{y}}\big(\mathbf{S}_1\notin \overline{D}_x\big)\sum_{k=1}^{\infty}p_k e^{-\theta k}
		+\mathbb{E}_{\mathbf{y}}\Big( 1_{\{\mathbf{S}_1\in \overline{D}_x\}} \sum_{k=1}^{\infty}p_k v_{(\gamma, \theta)}(\mathbf{S}_1;x)^k\Big)
		\bigg).
	\end{align}
	It follows from the definition of $\phi$ in \eqref{def-phi}, $\sum_{k=0}^\infty p_k=1$ and the above equation that 
	\begin{align}\label{eq-u-eta-theta}
		\quad \quad u_{(\gamma, \theta)}(\mathbf{y};x)
		=&1-e^{-\gamma}p_0
		-e^{-\gamma}\mathbb{P}_{\mathbf{y}}\left(\tau_x=1\right)\sum_{k=1}^{\infty}p_k e^{-\theta k} \nonumber \\
		&-e^{-\gamma}\mathbb{E}_{\mathbf{y}}\big( 1_{\{\mathbf{S}_1\in \overline{D}_x\}} \left[\phi( u_{(\gamma, \theta)}(\mathbf{S}_1;x))+(1-p_0)-u_{(\gamma, \theta)}(\mathbf{S}_1;x)\right]
		\big)\nonumber \\
		=&(1-e^{-\gamma})+e^{-\gamma} \mathbb{P}_{\mathbf{y}}\left(\tau_x=1\right) \sum_{k=1}^{\infty}p_k (1-e^{-\theta k})\nonumber \\
		&-e^{-\gamma}\mathbb{E}_{\mathbf{y}}\big( 1_{\{\mathbf{S}_1\in \overline{D}_x\}} \phi\left(u_{(\gamma, \theta)}(\mathbf{S}_1;x)\right)\big)
		+e^{-\gamma}\mathbb{E}_{\mathbf{y}}\big( 1_{\{\mathbf{S}_1\in \overline{D}_x\}} u_{(\gamma, \theta)}(\mathbf{S}_1;x)\big) \nonumber \\
		=:&H_{(\gamma,\theta)}(\mathbf{y};x)+e^{-\gamma}\mathbb{E}_{\mathbf{y}}\big( 1_{\{\mathbf{S}_1\in \overline{D}_x\}} u_{(\gamma, \theta)}(\mathbf{S}_1;x)\big).
	\end{align}
	Replacing $\mathbf{y}$ with $\mathbf{S}_1$ in \eqref{eq-u-eta-theta} and using the Markov property, we see that on the set $\{\mathbf{S}_1\in \overline{D}_x\}$, 
	\begin{align}\label{eq-u-eta-theta-1}
		u_{(\gamma, \theta)}(\mathbf{S}_1;x)
		=H_{(\gamma,\theta)}(\mathbf{S}_1;x)+e^{-\gamma}\mathbb{E}_{\mathbf{y}}\big( 1_{\{\mathbf{S}_2\in \overline{D}_x\}} u_{(\gamma, \theta)}(\mathbf{S}_2;x)\big|\mathcal{F}_1\big),
	\end{align}
	where $\mathcal{F}_n:=\sigma(\mathbf{S}_0,...,\mathbf{S}_n)$. Combining \eqref{eq-u-eta-theta} and \eqref{eq-u-eta-theta-1}, we get
	\begin{align}
		&u_{(\gamma, \theta)}(\mathbf{y};x)
		=H_{(\gamma,\theta)}(\mathbf{S}_0;x)
		+e^{-\gamma}\mathbb{E}_{\mathbf{y}}\big(1_{\{\mathbf{S}_1\in \overline{D}_x\}} H_{(\gamma,\theta)}(\mathbf{S}_1;x)\big)
		+e^{-2\gamma}\mathbb{E}_{\mathbf{y}}\left(1_{\{\tau_x>2\}} u_{(\gamma, \theta)}(\mathbf{S}_2;x)\right).
	\end{align}
	According to the definition of $H_{(\gamma, \theta)}$, it is easy to see that 
	\begin{align}
		&  H_{(\gamma,\theta)}(\mathbf{S}_0;x)
		+e^{-\gamma}\mathbb{E}_{\mathbf{y}}\big(1_{\{\mathbf{S}_1\in \overline{D}_x\}} H_{(\gamma,\theta)}(\mathbf{S}_1;x)\big)\\
		& = (1-e^{-\gamma})+e^{-\gamma} \mathbb{P}_{\mathbf{y}}\left(\tau_x=1 \right) \sum_{k=1}^{\infty}p_k (1-e^{-\theta k}) -e^{-\gamma}\mathbb{E}_{\mathbf{y}}\left( 1_{\{\tau_x>1 \}} \phi\left(u_{(\gamma, \theta)}(\mathbf{S}_1;x)\right)\right) \\
		& \quad + (1-e^{-\gamma})e^{-\gamma} \mathbb{P}_{\mathbf{y}}(\tau_x>1) +e^{-2\gamma} \mathbb{P}_{\mathbf{y}}\left(\tau_x=2 \right) \sum_{k=1}^{\infty}p_k (1-e^{-\theta k}) \\ &\quad -e^{-2\gamma}\mathbb{E}_{\mathbf{y}}\left( 1_{\{\tau_x>2 \}} \phi\left(u_{(\gamma, \theta)}(\mathbf{S}_2;x)\right)\right).
	\end{align}
	Therefore, we conclude that when $\mathbf{y}\in \overline{D}_x$, 
	\begin{align}
		u_{(\gamma, \theta)}(\mathbf{y};x)
		& =(1-e^{-\gamma}) \sum_{i=0}^1 e^{-\gamma i} \mathbb{P}_{\mathbf{y}}(\tau_x>i) +\sum_{k=1}^{\infty}p_k (1-e^{-\theta k}) \sum_{i=1}^2 e^{-\gamma i} \mathbb{P}_{\mathbf{y}}\left(\tau_x=i \right)  \\
		&\quad -\sum_{i=1}^2 e^{-\gamma i}\mathbb{E}_{\mathbf{y}}\left( 1_{\{\tau_x>i \}} \phi\left(u_{(\gamma, \theta)}(\mathbf{S}_i;x)\right)\right) 
		+e^{-2\gamma}\mathbb{E}_{\mathbf{y}}\left(1_{\{\tau_x>2\}} u_{(\gamma, \theta)}(\mathbf{S}_2;x)\right).
	\end{align}
	Iterating the above equation $N$ times, we deduce that 
	\begin{align}\label{expression-u}
		u_{(\gamma, \theta)}(\mathbf{y};x)
		& =(1-e^{-\gamma})\sum_{i=0}^{N-1}e^{-\gamma i} \mathbb{P}_{\mathbf{y}}\left(\tau_x>i\right)
		+\sum_{k=1}^{\infty}p_k (1-e^{-\theta k})\sum_{i=1}^{N}e^{-\gamma i} \mathbb{P}_{\mathbf{y}}\left(\tau_x=i\right) \nonumber \\
		&\quad -\sum_{i=1}^{N}e^{-\gamma i}\mathbb{E}_{\mathbf{y}}\left(1_{\{\tau_x>i\}}\phi \left(u_{(\gamma,\theta)}(\mathbf{S}_i;x)\right)\right)
		+e^{-\gamma N}\mathbb{E}_{\mathbf{y}}\left(1_{\{\tau_x>N\}} u_{(\gamma, \theta)}(\mathbf{S}_N;x)\right).
	\end{align}
	Noticing that $\lim_{N\to\infty} \mathbb{P}_{\mathbf{y}}(\tau_x>N)=0$ and that $u_{(\gamma, \theta)}\in [0,1]$, the last term on the right hand side of \eqref{expression-u} vanishes as $N$ tends to infinity. Therefore,  we arrive at the desired result by taking $N\to\infty$.
\end{proof}

For any $\mathbf{y}\in \overline{D}_1,  \theta,\gamma\geq 0$ and $x>0$, define 
\begin{align}\label{Def-of-phi-x}
	U_{(\gamma,\theta)}\left(\mathbf{y};x\right):=x^\frac{2}{\alpha-1}u_{(\gamma x^{-\frac{2\alpha}{\alpha-1}},\theta x^{-\frac{2}{\alpha-1}})}\left(x\mathbf{y};x\right)\quad \mbox{and} \quad \phi^{(x)} (u):= x^{\frac{2\alpha}{\alpha-1}} \phi\big( ux^{-\frac{2}{\alpha-1}}\big) . 
\end{align}
The following lemma establishes fundamental properties of  $\phi^{(x)}$. The proof is standard and can be found in \cite[Lemma 2.8]{HRS2026} and \cite[ Lemma 2.14]{HRS2025}.

\begin{lemma}\label{lemma-property-phi}
	There exists a positive constant $C_\phi$ such that 
	\begin{align}\label{upper-bound-phi}
		|\phi^{(x)}(u)-\phi^{(x)}(v)|\le C_\phi~|u-v||u^{\alpha-1}+v^{\alpha-1}|,\quad \forall~u,v\in[0,x^{\frac{2}{\alpha-1}}],
	\end{align}
	and for each $K>0$, 
	\begin{align}\label{asymptotic-phi}
		\lim_{x\to \infty}\sup_{u\in [0,K]} \left| 
		u^{-\alpha}\phi^{(x)}(u)
		-\mathcal{C}(\alpha) \right| = 0,
	\end{align}
	where the constant $\mathcal{C}(\alpha)$ is defined in \eqref{Stable-branching}. 
\end{lemma}

Define
\begin{align}\label{Def-of-G-x}
	G^x(\mathbf{y}):=x^{\frac{2}{\alpha-1}}\Big(1-\mathbb{E}_{x\mathbf{y}}\big(e^{-\gamma  x^{-\frac{2\alpha}{\alpha-1}} \tau_x}\big)
	+\sum_{k=1}^{\infty}p_k \big(1-e^{-\theta x^{-\frac{2}{\alpha-1}} k}\big) \mathbb{E}_{x\mathbf{y}}\big(e^{-\gamma x^{-\frac{2\alpha}{\alpha-1}}  \tau_x}\big)\Big).
\end{align}
The following proposition, the main result of this section, provides an expression for $U_{(\gamma, \theta)}(\mathbf{y};x)$ and shows that for each fixed $\gamma, \theta \geq 0$, the function $U_{\theta}(\mathbf{y};x)$ is uniformly bounded.

\begin{prop}\label{lemma-expression-U}
	For any $x>1$ and $\mathbf{y}\in \overline{D}_1$, we have 
	\begin{align}
		U_{(\gamma,\theta)}(\mathbf{y};x)  = G^x (\mathbf{y})
		-x^{-2} \mathbb{E}_{x\mathbf{y}}\bigg(\sum_{i=1}^{\tau_x-1}e^{-\gamma i x^{-\frac{2\alpha}{\alpha-1}}} \phi^{(x)}\Big(U_{(\gamma, \theta)}\Big(\frac{\mathbf{S}_{i}}{x};x\Big)\Big)\bigg),
	\end{align}
	here we used the convension that $\sum_{i=1}^0 =0$ on the set $\{\tau_x=1\}$. 
	Moreover, for any $\theta, \gamma\geq 0 $,
	\begin{align}\label{def-Gamma}
		\sup_{\mathbf{y}\in \overline{D}_1, x>1} U_{(\gamma, \theta)}(\mathbf{y};x)  \leq  C_\tau \gamma 
		+\theta=: \Gamma_{(\gamma, \theta)}<\infty.
	\end{align}
\end{prop}
\begin{proof}
	Combining Lemma \ref{expression-small-u} and the definitions of 
	$U_{(\gamma, \theta)}$ 
	and $\phi^{(x)}$, we have 
	\begin{align}\label{expression-U}
		U_{(\gamma,\theta)}\left(\mathbf{y};x\right)
		&=	G^x(\mathbf{y})-x^{\frac{2}{\alpha-1}}\sum_{i=1}^{\infty}e^{-\gamma i x^{-\frac{2\alpha}{\alpha-1}} }\mathbb{E}_{x \mathbf{y}}\bigg(1_{\{\tau_x>i\}}\phi \Big(x^{-\frac{2}{\alpha-1}}U_{(\gamma,\theta)}\Big(\frac{\mathbf{S}_i}{x};x\Big)\Big)\bigg)\\
		& = G^x(\mathbf{y})-\frac{1}{x^2}\sum_{i=1}^{\infty}e^{-\gamma i x^{-\frac{2\alpha}{\alpha-1}} }\mathbb{E}_{x \mathbf{y}}\bigg(1_{\{\tau_x>i\}}\phi^{(x)} \Big(U_{(\gamma,\theta)}\Big(\frac{\mathbf{S}_i}{x};x\Big)\Big)\bigg),
	\end{align}
	which implies the first result.
	For the second inequality,  since  $\phi(u)\geq 0$, we conclude from the inequality $1-e^{-z}\leq z$ for all $z\geq 0$ and the definition of $G^x$ that
	\begin{align}\label{upp-U}
		& U_{(\gamma, \theta)}(\mathbf{y};x) 
		\leq  x^{\frac{2}{\alpha-1}}\bigg(1-\mathbb{E}_{x\mathbf{y}}\Big(e^{-\gamma  x^{-\frac{2\alpha}{\alpha-1}} \tau_x}\Big)
		+\sum_{k=1}^{\infty}p_k \Big(1-e^{-\theta x^{-\frac{2}{\alpha-1}} k}\Big) \mathbb{E}_{x\mathbf{y}}\Big(e^{-\gamma x^{-\frac{2\alpha}{\alpha-1}}  \tau_x}\Big)\bigg)\\
		& \leq  x^{\frac{2}{\alpha-1}}\Big(\gamma x^{-\frac{2\alpha}{\alpha-1}}  \mathbb{E}_{x\mathbf{y}}\left(\tau_x\right)
		+\theta x^{-\frac{2}{\alpha-1}}  \sum_{k=1}^{\infty} k p_k  \Big) 
		= \gamma x^{-2}  \mathbb{E}_{x\mathbf{y}}\left(\tau_x\right)
		+\theta \leq C_\tau \gamma 	+\theta,
	\end{align}
	where in the last equality we used the fact that $\sum_{k=1}^\infty kp_k =1$ and the last inequality follows from Lemma \ref{lemma-step1}. This gives the desired result.
\end{proof}

\section{Proof of the main results}\label{S4}

\subsection{Convergence of $U_{(\gamma,\theta)}$ for small $\gamma$ and $\theta$}\label{S4.1}

Let $\theta_0>0$ be a fixed constant such that 
\begin{align}\label{Def-of-theta-0}
	2C_\tau C_\phi \Gamma_{(\theta_0, \theta_0)}^{\alpha-1} =\frac{1}{2},
\end{align}
where $\Gamma_{(\theta_0, \theta_0)}$ is defined in \eqref{def-Gamma}.
From Proposition \ref{lemma-expression-U}, for each fixed $\gamma, \theta\in [0,\theta_0]$, the function $U_{(\gamma, \theta)} (\mathbf{y} ; x)$ is bounded for all $\mathbf{y}\in \overline{D}_1$ and $x>1$. 
Hence, for any fixed $\gamma, \theta>0$ and any sequence $\{a_j\}\subset [1,\infty)$ satisfying  $\lim_{j\to \infty}a_j= \infty$, according to a diagonalization argument, there exists a subsequence $\{x_k\}=\{x_k(\gamma, \theta)\} \subset  \{a_j\}$ with $\lim_{k\to\infty}x_k=\infty$ such that the following limit exists:
\begin{align}\label{limi-of-U}
	V_{(\gamma,\theta)}(\mathbf{y}):= \lim_{k\to\infty} U_{(\gamma, \theta)} (\mathbf{y};x_k), \quad \forall \mathbf{y}\in \overline{D}_1\cap \mathbb{Q}^d. 
\end{align}
Our next result shows that the above limit holds for all $ \mathbf{y} \in \overline{D}_1$.

\begin{prop}
	$(1)$ $V_{(\gamma,\theta)}(\mathbf{y})$ is continuous in $\mathbf{y} \in \overline{D}_1 \cap \mathbb{Q}^d$. Moreover, the following limit exists:
	\begin{align}
		V_{(\gamma,\theta)}(\mathbf{y}):=\lim_{\mathbf{z}\to \mathbf{y}, \mathbf{z} \in \overline{D}_1 \cap \mathbb{Q}^d} V_{(\gamma,\theta)}(\mathbf{z}), \quad \mathbf{y} \in \overline{D}_1.
	\end{align}
	$(2)$ The limit \eqref{limi-of-U} holds for any $\mathbf{y} \in \overline{D}_1$.
\end{prop}
\begin{proof}
	(1) We treat $G^x$ first. Recall the definition of $\tau_x^{\mathbf{z}}$ in \eqref{Def-tau-x-z}. Since $(\tau_x, \mathbb{P}_{x(\mathbf{y}+\mathbf{z})}) \stackrel{\mathrm{d}}{=}  (\tau_x^\mathbf{z}, \mathbb{P}_{x\mathbf{y}}) $,  
	for each $x>1$ and $\mathbf{y}, \mathbf{y}+\mathbf{z}\in \overline{D}_1$,
	\begin{align}\label{upper-bound-G-x}
		\left| G^x(\mathbf{y})- G^x(\mathbf{y}+\mathbf{z})\right| & \leq  x^{\frac{2}{\alpha-1}} \Big|\mathbb{E}_{x\mathbf{y}}\Big(e^{-\gamma x^{-\frac{2\alpha}{\alpha-1} \tau_x}} - e^{-\gamma x^{-\frac{2\alpha}{\alpha-1} \tau_x^\mathbf{z}}} \Big)\Big| 
		\nonumber \\
		&\quad +x^{\frac{2}{\alpha-1}}\sum_{k=1}^{\infty}p_k \Big(1-e^{-\theta x^{-\frac{2}{\alpha-1}}k}\Big)\Big|\mathbb{E}_{x\mathbf{y}}\Big(e^{-\gamma x^{-\frac{2\alpha}{\alpha-1}} \tau_x}- e^{-\gamma x^{-\frac{2\alpha}{\alpha-1}} \tau_x^\mathbf{z}}\Big)\Big| \nonumber \\
		& \leq  \gamma x^{- 2} \mathbb{E}_{x\mathbf{y}}\left( \left| \tau_x-  \tau_x^\mathbf{z} \right| \right)
		+\theta \gamma x^{-\frac{2\alpha}{\alpha-1}}\sum_{k=1}^{\infty}k p_k \mathbb{E}_{x\mathbf{y}}\left( \left| \tau_x-  \tau_x^\mathbf{z} \right| \right) \nonumber \\
		& = \big(\gamma+  \theta \gamma x^{-\frac{2}{\alpha-1}}\big) x^{- 2} \mathbb{E}_{x\mathbf{y}}\left( \left| \tau_x-  \tau_x^\mathbf{z} \right| \right)
		\leq \left(\gamma+  \theta \gamma \right) x^{- 2} \mathbb{E}_{x\mathbf{y}}\left( \left| \tau_x-  \tau_x^\mathbf{z} \right| \right),
	\end{align}
	where the second inequality follows from 
	the fact  $|e^{-a}-e^{-b}|\leq |a-b|$ for $a,b\geq 0$  and  we use the fact that $\sum_{k=1}^\infty kp_k =1$ in the last equallity.

	Now we turn to the proof of (1).  Combining Proposition \ref{lemma-expression-U} and \eqref{upper-bound-G-x},
	\begin{align}\label{difference-U}
		&\big| U_{(\gamma,\theta)}\left(\mathbf{y};x\right) -U_{(\gamma,\theta)}\left(\mathbf{y}+\mathbf{z};x\right)\big| \nonumber \\
		\le & \left(\gamma+  \theta \gamma \right) x^{- 2} \mathbb{E}_{x\mathbf{y}}\left( \left| \tau_x-  \tau_x^\mathbf{z} \right| \right)
		+x^{-2}\bigg| \mathbb{E}_{x \mathbf{y}}\bigg(\sum_{i=1}^{\tau_{x}-1}e^{-\gamma ix^{-\frac{2\alpha}{\alpha-1}}}\phi^{(x)} \Big(U_{(\gamma,\theta)}\Big(\frac{\mathbf{S}_i}{x};x\Big)\Big)\bigg) \nonumber
		\\
		&\quad -\mathbb{E}_{x \mathbf{y}}\bigg(\sum_{i=1}^{\tau_{x}^\mathbf{z}-1}e^{-\gamma ix^{-\frac{2\alpha}{\alpha-1}}}\phi^{(x)} \Big(U_{(\gamma,\theta)}\Big(\frac{\mathbf{S}_i}{x} +\mathbf{z};x\Big)\Big)\bigg)\bigg| \nonumber \\
		= : &\left(\gamma+  \theta \gamma \right) x^{- 2} \mathbb{E}_{x\mathbf{y}}\left( \left| \tau_x-  \tau_x^\mathbf{z} \right| \right)+ \Delta_{x} (\mathbf{y}, \mathbf{z}). 
	\end{align}
	Define
	\begin{align}
		G(\mathbf{z};x):=
		\sup_{\mathbf{a}, \mathbf{a}+\mathbf{z}\in \overline{D}_1}
		\left|U_{(\gamma,\theta)}\left(\mathbf{a};x\right)-U_{(\gamma,\theta)}\left(\mathbf{a}+\mathbf{z};x\right)\right|.
	\end{align}
	Combining  Lemma \ref{lemma-property-phi} and Proposition \ref{lemma-expression-U}, we get that 
	\begin{align}
		\Delta_{x} (\mathbf{y}, \mathbf{z}) & \leq x^{-2} C_\phi \mathbb{E}_{x \mathbf{y}}\bigg(\sum_{i=\tau_{x}\land \tau_{x}^\mathbf{z}-1}^{\tau_{x}-1}e^{-\gamma ix^{-\frac{2\alpha}{\alpha-1}}} \Big(U_{(\gamma,\theta)}\Big(\frac{\mathbf{S}_i}{x};x\Big)\Big)^\alpha \bigg)\\
		&\quad + x^{-2}		C_{\phi}
		\mathbb{E}_{x \mathbf{y}}\bigg(\sum_{i=\tau_{x}\land \tau_{x}^\mathbf{z}-1}^{\tau_{x}^\mathbf{z}-1}e^{-\gamma ix^{-\frac{2\alpha}{\alpha-1}}} \Big(U_{(\gamma,\theta)}\Big(\frac{\mathbf{S}_i}{x} +\mathbf{z};x\Big)\Big)^\alpha \bigg)\\
		&\quad + x^{-2} 2C_\phi \Gamma_{(\gamma, \theta)}^{\alpha-1}	G(\mathbf{z};x) \mathbb{E}_{x \mathbf{y}}\bigg(\sum_{i=1}^{\tau_{x}^\mathbf{z} \land \tau_x-1}e^{-\gamma ix^{-\frac{2\alpha}{\alpha-1}}} \bigg)\\
		& \leq x^{-2} 2C_\phi \Gamma_{(\gamma, \theta)}^{\alpha} \mathbb{E}_{\delta_{x \mathbf{y}}}\big(| \tau_x -\tau_x^{\mathbf{z}} |\big) 
		+ x^{-2} 2C_\phi \Gamma_{(\gamma, \theta)}^{\alpha-1}	G(\mathbf{z};x) \mathbb{E}_{x \mathbf{y}} \left(\tau_x\right). 
	\end{align}
	Therefore, combining the above inequality and Lemma \ref{lemma-step1}, we conclude that 
	\begin{align}
		\Delta_{x} (\mathbf{y}, \mathbf{z}) &\leq  2
		C_\phi \Gamma_{(\gamma, \theta)}^{\alpha} x^{-2} \mathbb{E}_{\delta_{x \mathbf{y}}}\big(| \tau_x -\tau_x^{\mathbf{z}}|\big) +  2 C_\tau C_\phi \Gamma_{(\gamma, \theta)}^{\alpha-1}	G(\mathbf{z};x) .
	\end{align}
	Plugging the inequality back to \eqref{difference-U} and taking 
	$\theta, \gamma \in [0, \theta_0]$,
	we get
	\begin{align}
		&\big| U_{(\gamma,\theta)}\left(\mathbf{y};x\right)-U_{(\gamma,\theta)}\left(\mathbf{y}+\mathbf{z};x\right)\big|\\
		& \le 2 C_\tau C_\phi \Gamma_{(\theta_0, \theta_0)}^{\alpha-1} G(\mathbf{z};x) 
		+\big[\theta_0+  \theta_0^2 + 2
		C_\phi \Gamma_{(\theta_0, \theta_0)}^{\alpha}  \big] \frac{1}{x^2}\mathbb{E}_{\delta_{x \mathbf{y}}}\big(| \tau_x -\tau_x^{\mathbf{z}}|\big),
	\end{align}
	where in the inequality we also used the fact that $\Gamma_{(\gamma,\theta)}$ is increasing in $\gamma$ and $\theta$. According to the definition of $\theta_0$ in \eqref{Def-of-theta-0}, taking supremum over $\mathbf{y}$ in the above inequality, we deduce that for any $x>1$ and $\gamma,\theta\in [0,\theta_0]$, 
	\begin{align}\label{upper-bound-sup-difference-U}
		G(\mathbf{z};x)  
		&\le 2\big[\theta_0+  \theta_0^2 + 2
		C_\phi \Gamma_{(\theta_0, \theta_0)}^{\alpha}  \big] \sup_{\mathbf{y},\mathbf{y}+\mathbf{z}\in \overline{D}_1} \frac{1}{x^2} \mathbb{E}_{\delta_{x \mathbf{y}}}\big(| \tau_x -\tau_x^{\mathbf{z}}|\big)\nonumber \\
		& =: C_1(\theta_0) \sup_{\mathbf{y},\mathbf{y}+\mathbf{z}\in \overline{D}_1} \frac{1}{x^2} \mathbb{E}_{\delta_{x \mathbf{y}}}\big(| \tau_x -\tau_x^{\mathbf{z}}|\big).
	\end{align}
	By Lemma \ref{lemma-step2} (with $\delta= \| \mathbf{z}\|$), taking $x=x_k\to \infty$ in \eqref{upper-bound-sup-difference-U}, we see that for each $\mathbf{y}, \mathbf{z} \in \mathbb{Q}^d$ with $\mathbf{y}, \mathbf{y}+\mathbf{z}\in \overline{D}_1$, 
	\begin{align}
		\left| V_{(\gamma,\theta)}(\mathbf{y})-V_{(\gamma,\theta)}(\mathbf{y}+\mathbf{z})\right|
		\le 
		9\eta^{-2}C_1(\theta_0)
		\|\mathbf{z}\|,
	\end{align}
	which implies that $V_{(\gamma,\theta)}(\mathbf{y})$ is continuous in $\mathbf{y} \in \overline{D}_1 \cap \mathbb{Q}^d$. Therefore, for any fixed $\gamma,\theta \in [0,\theta_0]$, we can extend the definition of $V_{(\gamma,\theta)}(\mathbf{y})$ from $\mathbf{y} \in \overline{D}_1 \cap \mathbb{Q}^d$ to $\mathbf{y} \in \overline{D}_1 $.
	
	(2) For any $\varepsilon>0$, let $\delta>0$ be small enough such that $9 \eta^{-2} C_1(\theta_0) \delta  <\frac{\varepsilon}{2}.$
	Then it follows from Lemma \ref{lemma-step2} that there exists $N=N(\delta)$ such that as $k>N$,
	\[
	\sup_{\mathbf{z}\in \overline{D}_\delta}\sup_{\mathbf{y},\mathbf{y}+\mathbf{z}\in \overline{D}_1}\frac{1}{x_k^2} \cdot \mathbb{E}_{x_k\mathbf{y}}\left(|\tau_{x_k}-\tau_{x_k}^\mathbf{z} |\right) 
	\leq 18 \eta^{-2} \delta. 
	\]
	Plugging this back to \eqref{upper-bound-sup-difference-U} yields that when 
	$k>N$,
	\begin{align}\label{e7'}
		\sup_{\mathbf{z}\in \overline{D}_\delta}\sup_{\mathbf{y}, \mathbf{y}+\mathbf{z}\in \overline{D}_1}\left|U_{(\gamma,\theta)}\left(\mathbf{y};x\right)-U_{(\gamma,\theta)}\left(\mathbf{y}+\mathbf{z};x\right)\right|
		\le C_1(\theta_0) \times 
		18\eta^{-2}\delta
		<\varepsilon.
	\end{align}
	Now for any $\mathbf{y} \in \overline{D}_1$, taking $\mathbf{y}_m\in \overline{D}_1 \cap \mathbb{Q}^d$ with $\mathbf{y}_m \to \mathbf{y}$ as $m\to \infty$, then we can fix large $m$ such that $\|\mathbf{y}_m -\mathbf{y}\|\le \delta$, the above inequality 
	combined with \eqref{limi-of-U} implie that 
	\begin{align}
		&\limsup_{k\to\infty}\left| V_{(\gamma,\theta)}(\mathbf{y})- U_{(\gamma,\theta)}(\mathbf{y};x_k)\right|
		\le  \left| V_{(\gamma,\theta)}(\mathbf{y})-V_{(\gamma,\theta)}(\mathbf{y}_m)\right|\\
		&\quad +\limsup_{k\to\infty}\left| V_{(\gamma,\theta)}(\mathbf{y}_m)- U_{(\gamma,\theta)}(\mathbf{y}_m;x_k)\right|
		+\limsup_{k\to\infty}\left| U_{(\gamma,\theta)}(\mathbf{y}_m;x_k)-U_{(\gamma,\theta)}(\mathbf{y};x_k)\right|\\
		&\le  \left| V_{(\gamma,\theta)}(\mathbf{y})-V_{(\gamma,\theta)}(\mathbf{y}_m)\right|+\varepsilon.
	\end{align}
	Taking $m\to\infty$ first and then $\varepsilon \to 0$ in the above inequality, we get the desired result immediately. 
\end{proof}

Recall that $G^x(\mathbf{y})$ is defined in \eqref{Def-of-G-x} and  $\tau_{1}^W =\inf\{t>0: W_t \notin \overline{D}_1\}$. In the following lemma, we establish the equation of $V_{(\gamma,\theta)}$ for small parameters $\gamma,\theta\in [0,\theta_0]$.

\begin{prop}\label{lemma-invariance-principle}
	For any $\mathbf{y} \in \overline{D}_1$ and $\gamma , \theta \in  [0,\theta_0]$, it holds that 
	\begin{align}\label{equivalent-eq-V-eta-theta}
		V_{(\gamma,\theta)}(\mathbf{y})
		&=\gamma\, \mathbb{E}_{\mathbf{y}}\big(\tau_{1}^W\big)+\theta
		-\mathcal{C}(\alpha)\, \mathbb{E}_{\mathbf{y}}\bigg[\int_{0}^{\tau_{1}^W} V_{(\gamma,\theta)}^\alpha(W_s)\,\mathrm{d}s\bigg],
	\end{align}
	where  $\mathcal{C}(\alpha)$ is the constant in \eqref{Stable-branching}.
\end{prop}
\begin{proof}
	Since $\frac{2\alpha}{\alpha-1}>2$, combining Lemma \ref{lemma-step1} and the invariance principle, for each $\mathbf{y}\in \overline{D}_1$, we get 
	\begin{align}
		\lim_{x\to\infty} G^x(\mathbf{y})= 	\gamma
		\mathbb{E}_{\mathbf{y}}\big( \tau_1^W\big)
		+\theta. 
	\end{align}
	By Proposition \ref{lemma-expression-U},  $U_{(\gamma, \theta)}$ is uniformly bounded. 
	Let $\{x_k\}_{k \geq 0}$ be the sequence in \eqref{limi-of-U}. Combining  Lemma \ref{lemma-property-phi} and the inequality
	$e^{-\gamma ix_k^{-\frac{2\alpha}{\alpha-1}}} \geq 1- \gamma \tau_{x_k} x_k^{-\frac{2\alpha}{\alpha-1}}$
	for all $1\leq i\leq \tau_{x_k}$,
	we obtain that for any $\varepsilon>0$, there exists $N_\varepsilon>0$ such that when $x_k>N_\varepsilon$,
	\begin{align}\label{e8}
		& (1-\varepsilon)\mathcal{C}(\alpha)\sum_{i=1}^{\tau_{x_k}-1}
		\Big(U_{(\gamma, \theta)}\big(\tfrac{\mathbf{S}_i}{x_k};x_k\big)\Big)^\alpha -  \mathcal{C}(\alpha)\gamma \tau_{x_k}^2 x_k^{-\frac{2\alpha}{\alpha-1}}
		\Gamma_{(\gamma, \theta)}^\alpha\nonumber \\
		& \le (1-\varepsilon)\mathcal{C}(\alpha)\sum_{i=1}^{\tau_{x_k}-1}e^{-\gamma ix_k^{-\frac{2\alpha}{\alpha-1}}}
		\Big(U_{(\gamma, \theta)}\big(\tfrac{\mathbf{S}_i}{x_k};x_k\big)\Big)^\alpha \nonumber \\
		& \le  \sum_{i=1}^{\tau_{x_k}-1}e^{-\gamma ix_k^{-\frac{2\alpha}{\alpha-1}}}
		\phi^{(x_k)} \Big(U_{(\gamma, \theta)}\big(\tfrac{\mathbf{S}_i}{x_k};x_k\big)\Big) 
		=: \Upsilon_{x_k}^\mathbf{S}
		\le (1+\varepsilon)\mathcal{C}(\alpha)\sum_{i=1}^{\tau_{x_k}-1}
		\Big(U_{(\gamma, \theta)}\big(\tfrac{\mathbf{S}_i}{x_k};x_k\big)\Big)^\alpha.
	\end{align}
	Now we would like to replace function $U_{(\gamma, \theta)}$ by $V_{(\gamma, \theta)}$.  For each $\varepsilon>0$, we may choose a suitable $\delta>0$ such that \eqref{e7'} holds for all large $k$, which also holds with $U_{(\gamma, \theta)}$ replaced by $V_{(\gamma,\theta)}$.  Since $\overline{D}_1$ is compact, there exists a finite number of real numbers $\{\mathbf{y}_1,...,\mathbf{y}_L\}\subset \overline{D}_1$ such that $\overline{D}_1 = \cup_{j=1}^L (\overline{D}(\mathbf{y}_j, \delta)\cap \overline{D}_1).$ Therefore, assume that $k$ is large enough such that $\sup_{1\leq j\leq L} | V_{(\gamma,\theta)}(\mathbf{y}_j)- U_{(\gamma,\theta)}(\mathbf{y}_j; x_k)| <\varepsilon$, for any $\mathbf{y}\in \overline{D}_1$, suppose that $\mathbf{y}\in \overline{D}(\mathbf{y}_j, \delta)\cap \overline{D}_1$ for some $1\leq j \leq L$, then for large $k$,
	\begin{align}
		\left| V_{(\gamma,\theta)}(\mathbf{y})- U_{(\gamma,\theta)}(\mathbf{y}; x_k)\right| & \leq \left| V_{(\gamma,\theta)}(\mathbf{y})- V_{(\gamma,\theta)}(\mathbf{y}_j)\right| +	\left| V_{(\gamma,\theta)}(\mathbf{y}_j)- U_{(\gamma,\theta)}(\mathbf{y}_j; x_k)\right| \\
		&\qquad + \left| U_{(\gamma,\theta)}(\mathbf{y}_j; x_k)- U_{(\gamma,\theta)}(\mathbf{y}; x_k)\right| \leq 3\varepsilon. 
	\end{align}
	Plugging this back to \eqref{e8} implies that for large $k$,
	\begin{align}\label{e9}
		&	(1-\varepsilon)\mathcal{C}(\alpha)\sum_{i=1}^{\tau_{x_k}-1}
		\Big(\Big(V_{(\gamma,\theta)}\big(\tfrac{\mathbf{S}_i}{x_k}\big) - 3\varepsilon \Big)_+ \Big)^\alpha -  \mathcal{C}(\alpha)\gamma \tau_{x_k}^2 x_k^{-\frac{2\alpha}{\alpha-1}}
		\Gamma_{(\gamma, \theta)}^\alpha \nonumber \\
		& \le  
		\Upsilon_{x_k}^\mathbf{S}
		\le (1+\varepsilon)\mathcal{C}(\alpha)\sum_{i=1}^{\tau_{x_k}-1}
		\Big(V_{(\gamma,\theta)}\big(\tfrac{\mathbf{S}_i}{x_k}\big)+3\varepsilon \Big)^\alpha.
	\end{align}
	Combining \eqref{joint-convergence} and the continuous mapping theorem (see Billingsley \cite[Theorem 5.5]{Billingsley68}), 
	for any $f\in C_b^+(\mathbb{R}^d)$,
	\begin{align}\label{def-mathcal-S}
		\mathcal{S}_k(f):=\frac{1}{x_k^2}\sum_{i=1}^{\tau_{x_k}-1} f\left(\frac{\mathbf{S}_i}{x_k}\right)
		\xLongrightarrow{k\to\infty} 
		\int_{0}^{\tau_{1}^W}f(W_s)\mathrm{d}s,
	\end{align}
	where the convergence 
	holds
	in distribution.
	For a rigorous proof of the above statement, one can repeat the proof of Lemma 3.3 of Hou et al. \cite{HJRS2025} in the continuous-time setting.
	
	To show $\lim_{k\to\infty} \mathbb{E}_{x_k \mathbf{y}}(\mathcal{S}_k(f)) =  \mathbb{E}_{\mathbf{y}}\Big( \int_{0}^{\tau_{1}^W}f(W_s)\mathrm{d}s\Big)$, we need to prove that $(\mathcal{S}_k(f), \mathbb{P}_{x_k \mathbf{y}})$ is uniformly integrable, which 
	is proven in Lemma \ref{lemma-step1} since $\mathcal{S}_k(f)\leq \| f\|_\infty x_k^{-2} \tau_{x_k}$. Moreover, it follows from Lemma \ref{lemma-step1} and the fact $\frac{2\alpha}{\alpha-1}>2$ that 
	$\lim_{k\to\infty} x_k^{-2-\frac{2\alpha}{\alpha-1}}\mathbb{E}_{x_k \mathbf{y}}(\tau_{x_k}^2) =0. $
	Therefore, taking $f= (1-\varepsilon) \mathcal{C}(\alpha)(V_{(\gamma,\theta)}-3\varepsilon)_+^\alpha$ and $f= (1+\varepsilon) \mathcal{C}(\alpha)(V_{(\gamma, \theta)}+3\varepsilon)^\alpha$ respectively, it follows from \eqref{e9} that 
	\begin{align}
		&(1-\varepsilon)\mathcal{C}(\alpha) \mathbb{E}_\mathbf{y}\bigg[ \int_0^{\tau_{1}^W}
		\Big(\Big(V_{(\gamma, \theta)} \big(W_s\big) - 3\varepsilon \Big)_+ \Big)^\alpha \mathrm{d}s \bigg]  
		\le   \liminf_{k\to\infty}
		x_k^{-2}\,
		\mathbb{E}_{x_k \mathbf{y}}
		\big[\Upsilon_{x_k}^\mathbf{S}	\big]
		\\
		& \le   \limsup_{k\to\infty}
		x_k^{-2}\,
		\mathbb{E}_{x_k \mathbf{y}}
		\big[\Upsilon_{x_k}^\mathbf{S}	\big]
		\le  (1+\varepsilon)\mathcal{C}(\alpha)\mathbb{E}_\mathbf{y}\bigg[  \int_0^{\tau_{1}^W}
		\Big(V_{(\gamma, \theta)} \big(W_s\big)+3\varepsilon \Big)^\alpha\mathrm{d}s\bigg].
	\end{align}
	Now taking $\varepsilon \to 0$, we immediately get the desired result. 
\end{proof}

Proposition \ref{lemma-invariance-principle} shows the existence of the solution $V_{(\gamma,\theta)}$ to the equation in Proposition \ref{lemma-invariance-principle}, a natural question is about the uniqueness of the solution. Our next result answers this question. 
\begin{prop}\label{Uniqueness}
	
	(1) The solution $V_{(\gamma,\theta)}$ to 
	the equation \eqref{equivalent-eq-V-eta-theta} is unique. Moreover,
	\begin{align}\label{convegence-U-D1}
		V_{(\gamma,\theta)}(\mathbf{y})=\lim_{x\to\infty}U_{(\gamma,\theta)}(\mathbf{y};x), \quad \forall~\mathbf{y} \in \overline{D}_1,
		~ \gamma, \theta \in [0,\infty).
	\end{align} 
	
	(2) $V_{(\gamma,\theta)}$ is increasing in $\theta$ and $\gamma$ and set
	$V_{(\gamma,\infty)}(\mathbf{y}):=\lim_{\theta \to \infty} V_{(\gamma,\theta)} (\mathbf{y})$. Then 
	\begin{align}\label{expression-V-0-infty}
		V_{(0,\infty)}(\mathbf{y})= \mathbb{N}_\mathbf{y}\big(M^{X,d}\geq 1\big) 
	\end{align}
	is continuous in $\mathbf{y}\in D_1$.

	(3)
	$\lim_{\gamma\downarrow 0} V_{(\gamma,\infty)} (\mathbf{y}) = V_{(0,\infty)} (\mathbf{y}) .$

\end{prop}
\begin{proof}
	(1) Taking $D=\overline{D}_1, g=\theta$ and  $f= \gamma$ in \eqref{eq:Dynkin-integral-equation}, the equation \eqref{eq:Dynkin-integral-equation} reduces to the equation in Proposition \ref{lemma-invariance-principle}. Therefore, we have $V_{(\gamma,\theta)}(\mathbf y)=V_{f,g}^{\overline{D}_1}(\mathbf y)$ for all $\mathbf y\in \overline{D}_1$.
	Applying \eqref{eq:Dynkin-Laplace-functional} and \eqref{expression-of-Y} yields,  for any $\mathbf{y} \in \overline{D}_1$ and $\gamma, \theta \ge 0$,
	\begin{align}\label{expression-V-eta-theta}
		V_{(\gamma,\theta)}(\mathbf y)
		=&-\log \mathbb E_{\delta_{\mathbf y}}
		\Big[\exp\Big\{- \theta \langle 1, X^{\overline{D}_1} \rangle-\gamma \int_0^\infty \langle 1, X_t^{D_1}\rangle \mathrm{d}t \Big\}\Big].
	\end{align}
	Now the uniqueness implies that if there exists a subsequence $\{x_k\}_{k\geq 0} \subset [1,\infty)$ with $\lim_{k\to\infty} x_k=\infty$ such that $\lim_{k\to\infty} U_{(\gamma,\theta)}(\mathbf{y};x_k)$ exists for any $\mathbf{y}\in \overline{D}_1$, then the limit must be the same with representation given by \eqref{expression-V-eta-theta}, which implies that 
	\begin{align}
		V_{(\gamma,\theta)}(\mathbf{y})=\lim_{x\to\infty}U_{(\gamma,\theta)}(\mathbf{y};x), \quad \forall~\mathbf{y} \in \overline{D}_1,~\gamma,\theta \in [0,\theta_0].
	\end{align}
	Together with the branching property, we see that for any $\gamma, \theta \in [0,\theta_0]$ and $\mathbf{y} \in \overline{D}_1$,
	\begin{align}\label{eq-convergence}
		\lim_{x\to\infty} \mathbb{E}_{\lfloor x^{\frac{2}{\alpha-1}} \rfloor\delta_{x \mathbf{y}}}\Big(e^{-\gamma x^{-\frac{2\alpha}{\alpha-1}} \zeta_x-\theta x^{-\frac{2}{\alpha-1}} \langle 1, Z_{\overline{D}_x}\rangle} \Big)
		= \mathbb E_{\delta_{\mathbf y}}
		\Big[e^{- \theta \langle 1, X^{\overline{D}_1} \rangle-\gamma \int_0^\infty \langle 1, X_t^{D_1}\rangle \mathrm{d}t }\Big].
	\end{align}
	This combined with Cram\'{e}r--Wold theorem and Kallenberg \cite[exercise 9, p101]{Kallenberg} yields that 
	\begin{align}
		\bigg(\Big( \frac{\zeta_x}{x^{\frac{2\alpha}{\alpha-1}}}, \frac{\langle 1, Z_{\overline{D}_x}\rangle}{x^{\frac{2}{\alpha-1}}}\Big), \  \mathbb{P}_{\lfloor x^{\frac{2}{\alpha-1}} \rfloor \delta_{x \mathbf{y}}}\bigg)
		\xRightarrow[
		x \to \infty
		]{d}  \bigg( \Big(   \int_0^\infty \langle 1, X_t^{D_1}\rangle \mathrm{d}t , \langle 1, X^{\overline{D}_1} \rangle\Big),  \  \mathbb{P}_{\delta_{\mathbf{y}}}\bigg).
	\end{align}
	Thus, \eqref{eq-convergence} holds for any $\gamma, \theta \ge 0$, which together with the branching property implies (1).

	(3)
	For each $\varepsilon\in (0,1)$, 
	fix any $f \in C_b^+({\mathbb{R}^d})$
	such that $f(\mathbf{a})>0$ if and only if $ \| \mathbf{a}\|> 1-\varepsilon $.
	From Li \cite[Theorems 5.15 and 8.27]{LZbook}, we see that for any $\kappa>0$ and $\mathbf{y}\in \overline{D}_1$,
	\begin{align}\label{Appl-N-measure}
		-\log \mathbb{E}_{\delta_{\mathbf{y}}}\Big(\exp\Big\{-\kappa \int_0^\infty \langle f, X_t\rangle \mathrm{d}t\Big\}\Big)= \mathbb{N}_{\mathbf{y}}\Big(1-\exp\Big\{-\kappa \int_0^\infty \langle f, w_t\rangle \mathrm{d}t\Big\}\Big).
	\end{align}
	Taking $\kappa \uparrow \infty$ in the above equality and noticing that $\langle f, X_t\rangle $ and $\langle f, w_t\rangle$ are almost everywhere right-continuous under $\mathbb{P}_{\delta_\mathbf{y}}$ and $\mathbb N_\mathbf{y}$ 
	respectively,
	we obtain that 
	\begin{align}
		-\log \mathbb{P}_{\delta_{\mathbf{y}}}\big( X_t((\overline{D}_{1-\varepsilon})^c)=0, \forall t>0 \big)= \mathbb{N}_{\mathbf{y}}\big( \exists t>0, \ w_t((\overline{D}_{1-\varepsilon})^c)>0\big).
	\end{align}
	With a slight abuse of notation, we use the same notation $M^{X,d}$ under $\mathbb{P}_{\delta_\mathbf{y}}$ as that in \eqref{Def-of-maximal} with $w$ replaced by $X$, then we conclude from the above equation that 
	\begin{align}\label{e13}
		-\log \mathbb{P}_{\delta_{\mathbf{y}}}\big( M^{X,d}\leq 1-\varepsilon \big)= \mathbb{N}_{\mathbf{y}}\big( M^{X,d} >1-\varepsilon\big).
	\end{align}
	The monotonicity property of $V_{(\gamma,\theta)}$ is obvious by \eqref{expression-V-eta-theta}. Therefore, taking $\varepsilon\to 0$ in the above equation and together with the definition of the exit measure, we have
	\begin{align}
		V_{(0,\infty)}(\mathbf{y}) & =\lim_{\theta\uparrow +\infty} V_{(0,\theta)} (\mathbf{y})=-\log \mathbb P_{\delta_{\mathbf y}}
		\big(  \langle 1, X^{\overline{D}_1} \rangle =0\big)	=-\log \mathbb{P}_{\delta_{\mathbf{y}}}\big( M^{X,d}<1 \big)= \mathbb{N}_{\mathbf{y}}\big( M^{X,d} \geq 1\big).
	\end{align}
	Combining  $\mathbb{P}_{\delta_\mathbf{y}}\big(\int_0^\infty \langle 1, X_t^{D_1}\rangle \mathrm{d}t <\infty \big) =1$,
	dominated convergence theorem and \eqref{expression-V-eta-theta},
	\begin{align}
		\lim_{\gamma \to 0} V_{(\gamma, \infty)}(\mathbf{y})& = -\log \mathbb P_{\delta_{\mathbf y}}
		\Big( \langle 1, X^{\overline{D}_1} \rangle=0, \int_0^\infty \langle 1, X_t^{D_1}\rangle \mathrm{d}t <\infty\Big)\\ 
		&= -\log \mathbb P_{\delta_{\mathbf y}}
		\big( \langle 1, X^{\overline{D}_1} \rangle=0\big)
		= V_{(0, \infty)}(\mathbf{y}),
	\end{align}
	which implies 
	(3).

	(2) For the continuity of $V_{(0,\infty)}$,  
	let $\mathcal{L}$ be the generator of the Brownian motion $W$.
	According to Dynkin \cite[Theorems 1.2 and 1.3]{Dynkin91}, $V_{(0,\infty)}$ solves equation
	\begin{align}\label{PDE}
		-\mathcal{L} v + \mathcal{C}(\alpha) v^\alpha = 0 \quad \mbox{in}\  D_1,\quad 
		\lim_{\mathbf{y}\to \mathbf{a}\in \partial D_1}
		v(\mathbf{y})=\infty. 
	\end{align}
	Since the uniqueness for the solution to the above PDE and the continuity of the solution is proven in V\'eron \cite[Theorem 1.1]{Veron}, we complete the proof of (2).
\end{proof}

\subsection{Proof of Theorem \ref{thm-behavior-all-time-maximum}}\label{S4.2}

\begin{proof}[Proof of Theorem \ref{thm-behavior-all-time-maximum}]
	\emph{Step 1.} In this step, we prove that for any $\mathbf{y}\in D_1$,
	\begin{align}\label{Goal3}
		\liminf_{x\to\infty}x^{\frac{2}{\alpha-1}}\mathbb{P}_{\delta_{x\mathbf{y}}}(M^d\geq x)\ge 	\liminf_{x\to\infty}x^{\frac{2}{\alpha-1}}\mathbb{P}_{\delta_{x\mathbf{y}}}(M^d> x)\ge \mathbb{N}_\mathbf{y}\big(M^{X,d}\geq 1\big).
	\end{align}
	The first inequallity holds trivially, we only prove the second one here. 
	Combining inequality $1_{\{|z|>0\}}\ge 1-e^{-\theta |z|}$ for any $\theta>0$, $z\in \mathbb{R}$, 	  and the fact that $\mathbb{P}_{\delta_\mathbf{y}}(M^d>x)= \mathbb{P}_{\delta_\mathbf{y}}(\langle 1,Z_{\overline{D}_x}\rangle>0)$,  for any $\theta>0$ and $\mathbf{y}\in D_1$,
	\begin{align}
		x^{\frac{2}{\alpha-1}}\mathbb{P}_{\delta_{x\mathbf{y}}}(M^d> x)& =  x^{\frac{2}{\alpha-1}} \mathbb{P}_{\delta_{x\mathbf{y}}}(\langle 1,Z_{\overline{D}_x}\rangle>0)  \ge x^{\frac{2}{\alpha-1}}  \mathbb{E}_{\delta_{x\mathbf{y}}}\Big(1-e^{- \theta x^{-\frac{2}{\alpha-1}} \left\langle 1,Z_{\overline{D}_x}\right\rangle }\Big).
	\end{align}
	Combining Proposition \ref{Uniqueness} and the above inequality, we obtain that 
	\begin{align}
		&\liminf_{x\to\infty}x^{\frac{2}{\alpha-1}}\mathbb{P}_{\delta_{x\mathbf{y}}}(M^d> x) 
		\ge \liminf_{\theta\to\infty} \liminf_{x\to\infty}  x^{\frac{2}{\alpha-1}}  \mathbb{E}_{\delta_{x\mathbf{y}}}\Big(1-e^{- \theta x^{-\frac{2}{\alpha-1}} \left\langle 1,Z_{\overline{D}_x}\right\rangle }\Big) \\
		&=\liminf_{\theta\to\infty} \liminf_{x\to\infty}U_{(0, \theta)}(\mathbf{y};x)
		=\lim_{\theta\to\infty}V_{(0, \theta)}(\mathbf{y})=\mathbb{N}_\mathbf{y}\big(M^{X,d}\geq 1\big).
	\end{align}
	This gives \eqref{Goal3}. 
	
	\emph{Step 2.} In this step, we prove the following upper bound. For any $\mathbf{y}\in D_1$, 
	\begin{align}\label{Goal4}
		\limsup_{x\to\infty}x^{\frac{2}{\alpha-1}}\mathbb{P}_{\delta_{x\mathbf{y}}}(M^d> x)\leq \limsup_{x\to\infty}x^{\frac{2}{\alpha-1}}\mathbb{P}_{\delta_{x\mathbf{y}}}(M^d\geq x)\le \mathbb{N}_\mathbf{y}\big(M^{X,d}\geq 1\big).
	\end{align}
	The first inequality holds trivally, so we prove the second one here. 
	For any $\varepsilon  \in (0, \frac{1}{4} (1-\| \mathbf{y}\|) )$, since
	\[
	\big\{\langle 1,Z_{\overline{D}_{x(1-\varepsilon/4)}}\rangle =0\big\} = \{
	M^d \leq 
	x(1-\varepsilon/4)\}\subset  \{M^d<x\},
	\] 
	by the strong Markov property of $\{Z_{\overline{D}_x}\}_{x\ge 0}$, we get 
	\begin{align}\label{condition-expectation}
		&  \mathbb{E}_{\delta_{x\mathbf{y}}}\Big(1_{\left\{ M^d<x \right\}}\big|Z_{\overline{D}_{x(1-\varepsilon)}}\Big)\geq \mathbb{E}_{\delta_{x\mathbf{y}}}\Big(1_{\big\{\langle 1,Z_{\overline{D}_{x(1-\varepsilon/4)}}\rangle =0\big\}}\Big|Z_{\overline{D}_{x(1-\varepsilon)}}\Big) \nonumber \\
		&\quad \ge  1_{\big\{Z_{\overline{D}_{x(1-\varepsilon)}}\left(D_{x(1-\varepsilon/2)}^c\right)=0\big\}}
		\times \exp\Big\{
		\int_{\overline{D}_{x(1-\varepsilon)}^c}
		\log \mathbb{P}_{\delta_{\mathbf{a}}}\big( \langle 1,Z_{\overline{D}_{x(1-\varepsilon/4)}} \rangle=0\big) Z_{\overline{D}_{x(1-\varepsilon)}}(\mathrm{d} \mathbf{a}) \ \Big\}.
	\end{align}
	Since on the event $\big\{Z_{\overline{D}_{x(1-\varepsilon)}}\big(D_{x(1-\varepsilon/2)}^{c}\big)=0\big\}$, we have $x(1-\varepsilon)\le \|V(u)\|\le x(1-\varepsilon/2)$ for every $u\in \mathcal{T}_{\overline{D}_x}$. Therefore, it follows that 
	\begin{align}\label{lower-bound-1}
		&\mathbb{P}_{\delta_{V(u)}}\big( \langle 1,Z_{\overline{D}_{x(1-\varepsilon/4)}} \rangle=0\big)  
		=\mathbb{P}_{\delta_{V(u)}}\big(M^d\le x(1-\varepsilon/4) \big)
		\ge \mathbb{P}_{\delta_{V(u)}}\big(M^d< x(1-\varepsilon/4) \big)\nonumber \\
		&\quad = 1-\mathbb{P}_{\delta_{V(u)}}\big(M^d\geq x(1-\varepsilon/4)\big)
		=1-\mathbb{P}_{\delta_0}\Big(\sup_{n\ge 0}\max_{|w|=n}\|V(w)+\mathbf{a}\|\geq x(1-\varepsilon/4) \Big)\Big|_{\mathbf{a}=V(u)}\nonumber \\
		&\quad \ge 1-\mathbb{P}\big(M^d\geq x(1-\varepsilon/4)-\|\mathbf{a}\|\big)\big|_{\mathbf{a}=V(u)}
		\ge 1-\mathbb{P}\big(M^d>
		\varepsilon x/4\big).
	\end{align}
	Moreover, by \eqref{rough-behavior-M}, for $x$ large enough there exists a constant $c_*$ such that 
	\begin{align}
		\mathbb{P}\big(M^d\geq
		\varepsilon x/4\big)
		\le \frac{4^{\frac{2}{\alpha-1}}c_*}{\varepsilon^{\frac{2}{\alpha-1}} x^{\frac{2}{\alpha-1}}}=: \frac{C_*(\varepsilon) }{ x^{\frac{2}{\alpha-1}}}.
	\end{align}
	Combining this with \eqref{lower-bound-1}, we get that for $x$ large enough, 
	\begin{align}
		1\geq \mathbb{P}_{\delta_{V(u)}}\big(\langle 1,Z_{\overline{D}_{x(1-\varepsilon/4)}} \rangle=0 \big)
		\ge 1-  C_*(\varepsilon) x^{-\frac{2}{\alpha-1}} \quad  \text{on} \ \big\{Z_{\overline{D}_{x(1-\varepsilon)}}\big( D_{x(1-\varepsilon/2)}^{c}\big)=0\big\}.
	\end{align}
	Moreover, since $\lim_{z\to 1}\frac{\log z}{1-z}=-1$, it follows that for some $\delta>0$, when $1-z<\delta$, we have $\frac{\log z}{1-z}\ge -2$. Thus, for $x$ sufficiently large  such that $C_*(\varepsilon)/ x^{\frac{2}{\alpha-1}}<\delta$, we obtain that  for each $u\in \mathcal{T}_{\overline{D}_x}$, 
	\begin{align}
		\log \mathbb{P}_{\delta_{V(u)}}\big( \langle 1,Z_{\overline{D}_{x(1-\varepsilon/4)}} \rangle=0\big)  
		\ge \log \Big(  1-\frac{C_*(\varepsilon)}{x^{\frac{2}{\alpha-1}}}\Big)
		\ge -\frac{2C_*(\varepsilon)}{x^{\frac{2}{\alpha-1}}},\quad  \text{on} \ \big\{Z_{\overline{D}_{x(1-\varepsilon)}}\big(D_{x(1-\varepsilon/2)}^{c}\big)=0\big\}.
	\end{align}
	This combined with \eqref{condition-expectation} gives that 
	\begin{align}
		&\mathbb{E}_{\delta_{x\mathbf{y}}}\Big(1_{\left\{ M^d<x \right\}}\big|Z_{\overline{D}_{x(1-\varepsilon)}}\Big) 
		\ge 1_{\big\{Z_{\overline{D}_{x(1-\varepsilon)}}\left(D_{x(1-\varepsilon/2)}^c\right)=0\big\}}
		\times \exp\Big\{ -\frac{2C_*(\varepsilon)}{  x^{\frac{2}{\alpha-1}}}\langle 1 , Z_{\overline{D}_{x(1-\varepsilon)}} \rangle \Big\}\\
		&\ge \exp\Big\{ -\frac{2C_*(\varepsilon)}{  x^{\frac{2}{\alpha-1}}}\langle 1 , Z_{\overline{D}_{x(1-\varepsilon)}} \rangle \Big\}
		-1_{\big\{Z_{\overline{D}_{x(1-\varepsilon)}}\left(D_{x(1-\varepsilon/2)}^c\right)>0\big\}}.
	\end{align}
	Taking expectation in the above inequality implies that 
	\begin{align}
		\mathbb{P}_{\delta_{x\mathbf{y}}} (M^d<x)
		\ge &  \mathbb{E}_{\delta_{x\mathbf{y}}}\Big[\exp\Big\{ -\frac{2C_*(\varepsilon)}{  x^{\frac{2}{\alpha-1}}}\langle 1 , Z_{\overline{D}_{x(1-\varepsilon)}} \rangle \Big\}\Big]-\mathbb{P}_{\delta_{x\mathbf{y}}}\Big(Z_{\overline{D}_{x(1-\varepsilon)}}\big(D_{x(1-\varepsilon/2)}^c\big)>0\Big)\\
		=&\mathbb{E}_{\delta_{x\mathbf{y}}}\Big[\exp\Big\{ -\frac{2C_*(\varepsilon)}{ x^{\frac{2}{\alpha-1}}}\langle 1 , Z_{\overline{D}_{x(1-\varepsilon)}} \rangle \Big\}\Big]-\mathbb{P}_{\delta_{x\mathbf{y}}}\Big(Z_{\overline{D}_{x(1-\varepsilon)}}\big(D_{x(1-\varepsilon/2)}^c\big)\ge 1\Big).
	\end{align}
	Therefore, set $\theta_1:=2C_*(\varepsilon)(1-\varepsilon)^{\frac{2}{\alpha-1}} $,  by Markov inequality, 
	\begin{align}\label{upp-m-d-x}
		\ \ &x^{\frac{2}{\alpha-1}} \mathbb{P}_{\delta_{x\mathbf{y}}}(M^d\geq x)
		=x^{\frac{2}{\alpha-1}} \big(1-\mathbb{P}_{\delta_{x\mathbf{y}}}(M^d< x)\big)\nonumber \\
		&\le \frac{1}{(1-\varepsilon)^\frac{2}{\alpha-1}}U_{(0,\theta_1)}\Big(\frac{1}{1-\varepsilon}\mathbf{y}; x(1-\varepsilon)\Big)
		+x^{\frac{2}{\alpha-1}} \mathbb{P}_{\delta_{x\mathbf{y}}}\Big(Z_{\overline{D}_{x(1-\varepsilon)}}\big(D_{x(1-\varepsilon/2)}^c\big)\ge 1\Big)\nonumber \\
		&\le \frac{1}{(1-\varepsilon)^\frac{2}{\alpha-1}}U_{(0,\theta_1)}\Big(\frac{1}{1-\varepsilon}\mathbf{y}; x(1-\varepsilon)\Big)
		+x^{\frac{2}{\alpha-1}} \mathbb{E}_{\delta_{x\mathbf{y}}}\Big(Z_{\overline{D}_{x(1-\varepsilon)}}\big(D_{x(1-\varepsilon/2)}^c\big)\Big).
	\end{align}
	For the first term, by Proposition  \ref{Uniqueness},  we have 
	\begin{align}\label{asymptotic-behavior-I1}
		\lim_{x\to\infty}\frac{1}{(1-\varepsilon)^\frac{2}{\alpha-1}}U_{(0,\theta_1)}\Big(\frac{1}{1-\varepsilon}\mathbf{y}; x(1-\varepsilon)\Big)
		=\frac{ V_{(0,\theta_1)}\left(\frac{\mathbf{y}}{1-\varepsilon} \right)  }{(1-\varepsilon)^\frac{2}{\alpha-1}}
		\leq \frac{V_{(0,\infty)}\left(\frac{\mathbf{y}}{1-\varepsilon} \right)}{(1-\varepsilon)^\frac{2}{\alpha-1}}.
	\end{align}
	For the second term on the right hand side of \eqref{upp-m-d-x},  since $\frac{1}{1-\varepsilon} \mathbf{y}\in D_1$, combining many-to-one formula and Lemma \ref{Tightness}, we have
	\begin{align}\label{asymptotic-behavior-I2}
		&\limsup_{x\to\infty}x^{\frac{2}{\alpha-1}} \mathbb{E}_{\delta_{x\mathbf{y}}}\Big(Z_{\overline{D}_{x(1-\varepsilon)}}\big(D_{x(1-\varepsilon/2)}^c\big)\Big)\nonumber \\
		&= \limsup_{x\to\infty}x^{\frac{2}{\alpha-1}} \mathbb{P}_{x\mathbf{y}}\big(\|\mathbf{S}_{\tau_{x(1-\varepsilon)}}\|\geq x(1-\varepsilon/2)\big) \nonumber\\
		& \le  \limsup_{x\to\infty}x^{\frac{2}{\alpha-1}} \mathbb{P}_{x\mathbf{y}}\big(\|\mathbf{S}_{\tau_{x(1-\varepsilon)}}\|> x(1-\varepsilon/2)-1\big) \nonumber\\
		&\le 4 C_\tau (1-\varepsilon)^2 \limsup_{x\to\infty} x^{\frac{2\alpha}{\alpha-1}} \mathbb{P}\Big(\|\mathbf{X}\|> \frac{\varepsilon}{2}x-1\Big)=0,
	\end{align}
	where the last equality follows from our assumption $\mathbb{E}\big(\|\mathbf{X}\|^{\frac{2\alpha}{\alpha-1}}\big)<\infty$.
	Combining Proposition \ref{Uniqueness}, \eqref{asymptotic-behavior-I1} and \eqref{asymptotic-behavior-I2}, we obtain 
	\begin{align}
		\limsup_{x\to\infty} x^{\frac{2}{\alpha-1}} \mathbb{P}_{\delta_{x\mathbf{y}}}(M^d>x)
		& \le \frac{1}{(1-\varepsilon)^\frac{2}{\alpha-1}} V_{(0,\infty)}\Big(\frac{1}{1-\varepsilon}\mathbf{y} \Big)  \stackrel{\varepsilon \downarrow 0}{\longrightarrow}V_{(0,\infty)}\left(\mathbf{y} \right)
		=\mathbb{N}_\mathbf{y}\big(M^{X,d}\geq 1\big). 
	\end{align}
	This gives the  \eqref{Goal4}. 
	Combining \eqref{Goal3} and \eqref{Goal4}, we complete the proof of the theorem.
\end{proof}

\subsection{Proof of Theorem \ref{thm-behavior-Mn} and Corollary \ref{Cor}}\label{S4.3}

The classical Watanabe-type invariance principle 
for critical branching random walks states that (for a formal proof, see \cite[Theorem 1.2]{HRS2026}  with $h=0$ in $1$-dimensional continuous-time case and see \cite[Proposition 3.6]{HRS2025}  for $1$-dimensional branching killed L\'evy process)
for any $r>0$ and any 
$f\in C_b^+(\mathbb{R}^d)$,
\begin{align}\label{appl-invariance-principle}
	& \quad \lim_{n\to\infty} n^{\frac{1}{\alpha-1}} \left(1- \mathbb{E}_{\delta_{\sqrt{n}\mathbf{y}}} \left(\exp\left\{- \frac{1}{n^{\frac{1}{\alpha-1}}}\int f\Big(\frac{\mathbf{z}}{\sqrt{n}}\Big)Z_{\lfloor nr \rfloor}(\mathrm{d}\mathbf{z})\right\}\right)\right) = -\log \mathbb{E}_{\delta_{\mathbf{y}}}\Big(e^{-\langle f, X_r\rangle}\Big). 
\end{align}
Also, from \eqref{appl-invariance-principle}, we have the following scaling property: for any $t> \kappa>0$ and  
$f\in C_b^+(\mathbb{R}^d)$,
\begin{align}\label{scaling-property}
	\mathbb{E}_{\delta_{\mathbf{y}}}\Big(e^{-\langle f, X_{t\kappa}\rangle}\Big) = \mathbb{E}_{ \kappa^{-\frac{1}{\alpha-1}}\delta_{\mathbf{y}/\sqrt{\kappa}}}\Big( \exp\Big\{-\kappa^{\frac{1}{\alpha-1}} \int f(\sqrt{\kappa} \mathbf{z})X_t(\mathrm{d}\mathbf{z}) \Big\}\Big). 
\end{align}

With the help of the above properties, we are ready to adapt the ideas in \cite[Theorem 1.7]{HRS2025}  and prove Theorem \ref{thm-behavior-Mn}.

\begin{proof}[Proof of Theorem \ref{thm-behavior-Mn}]
	\emph{Step 1.} In this step, we show that
	$r> 0$, 
	\begin{align}\label{Goal1}
		\liminf_{n\to\infty} n^{\frac{1}{\alpha-1}} \mathbb{P}_{\delta_{\sqrt{n}\mathbf{y}}}\big(M_n^d >\sqrt{n}r\big)
		\ge \mathbb{N}_\mathbf{y}\big( M_1^{X} > r\big).
	\end{align}
	Define  
	$f(\mathbf{x}):= \min\{(\| \mathbf{x}\|-r)_+,1\}$,
	then $f\in C_b^+(\mathbb{R}^d)$
	and that for any non-negative measure $\nu$, $\langle f, \nu\rangle=0$ is equivalent to $\nu ((\overline{D}_r)^c)=0$.  Combining the inequality $1_{\{|z|>0\}}\geq 1-e^{-\theta |z|}$ and \eqref{appl-invariance-principle}, for any $\theta>0$, 
	\begin{align}\label{equality-Mn}
		&\liminf_{n\to\infty} n^{\frac{1}{\alpha-1}} \mathbb{P}_{\delta_{\sqrt{n}\mathbf{y}}}\big(M_n^d > \sqrt{n}r\big) 
		= \liminf_{n\to\infty}n^{\frac{1}{\alpha-1}} \mathbb{P}_{\delta_{\sqrt{n}\mathbf{y}}}\Big( \int_{\mathbb{R}^d} f(\mathbf{z}/\sqrt{n}) Z_n(\mathrm{d}\mathbf{z}) >0\Big)\\
		& \geq  \lim_{n\to\infty}n^{\frac{1}{\alpha-1}}\Big[1-\mathbb{E}_{\delta_{\sqrt{n}\mathbf{y}}}\Big(e^{-\theta n^{-\frac{1}{\alpha-1}}\int_{\mathbb{R}^d} f(\mathbf{z}/\sqrt{n}) Z_n(\mathrm{d}\mathbf{z}) }\Big) \Big]\\
		& =  -\log \mathbb{E}_{\delta_{\mathbf{y}}}\big(e^{-\theta \langle f, X_1\rangle}\big) 
		=  \mathbb{N}_\mathbf{y}\big(1-e^{-\theta \langle f, w_1\rangle}\big),
	\end{align}
	where in the last equality we used \eqref{N-measures}. Taking $\theta\to\infty$ in the above inequality  implies \eqref{Goal1}.
	
	\emph{Step 2.} In this step, we prove the upper bound: for any 
	$r> 0$, we have 
	\begin{align}\label{Goal2}
		\limsup_{n\to\infty} n^{\frac{1}{\alpha-1}} \mathbb{P}_{\delta_{\sqrt{n}\mathbf{y}}}\big(M_n^d > \sqrt{n}r\big)
		\le \mathbb{N}_\mathbf{y}\big( M_1^{X}>  r\big).
	\end{align}
	For each fixed $s\in (0,1)$, define $m_n:= \lfloor s n \rfloor$ and $k_n:= n-m_n$, according to the branching property, 
	\begin{align}\label{equality-Mn-1}
		& \mathbb{P}_{\delta_{\sqrt{n}\mathbf{y}}}\big(M_n^d > \sqrt{n}r\big)
		= 1-\mathbb{E}_{\delta_{\sqrt{n}\mathbf{y}}}\big( \mathbb{P}_{\delta_{\sqrt{n}\mathbf{y}}}(M_n^d \leq  \sqrt{n}r |\mathcal{F}_{k_n})\big) \nonumber \\
		& =  1- \mathbb{E}_{\delta_{\sqrt{n}\mathbf{y}}}\Big( \exp\Big\{ \sum_{|u|= k_n} \log \Big(1- \mathbb{P}_{\delta_{V(u)}}\left(M_{m_n}^d> \sqrt{n}r\right)\Big) \Big\} \Big)  .
	\end{align}
	Since for any $\mathbf{x}\in \mathbb{R}^d$, there exists a constant 
	$c_*$
	such that 
	(for example, see \cite[Lemma 2]{Slack68})
	\begin{align}\label{upper-bound-p-n}
		\mathbb{P}_{\delta_{\mathbf{x}}}\big(M_{m_n}^d> \sqrt{n}r\big)
		\le \mathbb{P}_{\delta_{\mathbf{x}}}\big(Z_{m_n} (\mathbb{R}^d)>0 \big)
		\le \frac{
			c_*
		}{m_n^{\frac{1}{\alpha-1}}}.
	\end{align}
	Therefore, uniformly for all $\mathbf{x}$, when $n$ is large enough, we have
	\[
	\log \Big(1- \mathbb{P}_{\delta_{\mathbf{x}}}\big(M_{m_n}^d> \sqrt{n}r\big)\Big)  \geq -\frac{1}{2} \mathbb{P}_{\delta_{\mathbf{x}}}\big(M_{m_n}^d> \sqrt{n}r\big). 
	\]
	Plugging this back to \eqref{equality-Mn-1} implies that when $n$ is large enough,
	\begin{align}\label{decompose-tail-probability-Mn}
		& \mathbb{P}_{\delta_{\sqrt{n}\mathbf{y}}}\big(M_n^d> \sqrt{n}r\big)
		\leq   1- \mathbb{E}_{\delta_{\sqrt{n}\mathbf{y}}}\Big( \exp\Big\{- \frac{1}{2} \sum_{|u|= k_n}  \mathbb{P}_{\delta_{V(u)}}\big(M_{m_n}^d> \sqrt{n}r\big) \Big\} \Big)  .
	\end{align}
	For each $\delta\in (0, r/2)$, noticing that  for any $\mathbf{x}\in \overline{D}_{\sqrt{n}(r-\delta)}$,
	\begin{align}
		\mathbb{P}_{\delta_{\mathbf{x}}}\big(M_{m_n}^d> \sqrt{n}r\big) & \leq  \mathbb{P}\big(M_{m_n}^d> \sqrt{n}r - \| \mathbf{x}\| \big)  \leq \mathbb{P}\big(M_{m_n}^d>\sqrt{n} \delta \big),
	\end{align}
	there exists at least one coordinate component (say $j$) such that $\max_{|w|= m_n} 
	|V^{(j)}(w)| 
	\ge \sqrt{n}\delta /d$. Therefore, by \cite[Lemma 2.6]{ZX25}  (or see \cite[Lemma 3.5]{HRS2026}  and \cite[Lemma 3.10]{HRS2025} ), there exists a constant $C_*= C_*(r,\delta,d)$ such that for large $n$, 
	\begin{align}\label{upper-bound-p-n-2}
		\mathbb{P}_{\delta_{\mathbf{x}}}\big(M_{m_n}^d> \sqrt{n}r\big) & \leq  
		s C_*n^{-\frac{1}{\alpha-1}}.
	\end{align}
	Combining \eqref{upper-bound-p-n}, \eqref{decompose-tail-probability-Mn} and \eqref{upper-bound-p-n-2}, we conclude that 
	\begin{align}
		\mathbb{P}_{\delta_{\sqrt{n}\mathbf{y}}}\big(M_n^d > \sqrt{n}r\big)		
		\leq & 1- \mathbb{E}_{\delta_{\sqrt{n}\mathbf{y}}}\bigg( \exp\Big\{- \frac{C_*sZ_{k_n}(\overline{D}_{\sqrt{n}(r-\delta)})}{2n^{\frac{1}{\alpha-1}}} -\frac{c_*Z_{k_n}((\overline{D}_{\sqrt{n}(r-\delta)})^c)}{2m_n^{\frac{1}{\alpha-1}}} \Big\} \bigg) \nonumber \\
		\leq &1- \mathbb{E}_{\delta_{\sqrt{n}\mathbf{y}}}\bigg( \exp\Big\{- \frac{C_*sZ_{k_n}(\mathbb{R}^d)}{2n^{\frac{1}{\alpha-1}}} -\frac{c_* Z_{k_n}((\overline{D}_{\sqrt{k_n}(r-\delta)})^c)}{2m_n^{\frac{1}{\alpha-1}}}  \Big\} \bigg) .
	\end{align}
	According to the  inequality $1-e^{-|a|-|b|}\leq 1-e^{-|a|} +|b|$ and noticing that $Z_n(\mathbb{R}^d)$ is a mean $1$ martingale, the above probability is bounded from above by
	\begin{align}\label{e11}
		\mathbb{P}_{\delta_{\sqrt{n}\mathbf{y}}}\big(M_n^d > \sqrt{n}r\big)
		\leq & 1- \mathbb{E}_{\delta_{\sqrt{n}\mathbf{y}}}\bigg( \exp\Big\{ -\frac{c_*Z_{k_n}((\overline{D}_{\sqrt{k_n}(r-\delta)})^c)}{2m_n^{\frac{1}{\alpha-1}}} \Big\} \bigg) + \frac{C_* s}{2n^{\frac{1}{\alpha-1}}} \nonumber \\
		\leq & 1- \mathbb{E}_{\delta_{\sqrt{n}\mathbf{y}}}\bigg( \exp\Big\{ -\frac{c_*}{2m_n^{\frac{1}{\alpha-1}}} \int f_\delta\left(\frac{\mathbf{z}}{\sqrt{k_n}}\right) Z_{k_n}(\mathrm{d}\mathbf{z}) \Big\} \bigg) + \frac{C_* s}{2n^{\frac{1}{\alpha-1}}},
	\end{align}
	where 
	$f_\delta(\mathbf{x}):= \min\{(\| \mathbf{x}\|-r+2\delta)_+,\delta\}/\delta
	\in C_b^+(\mathbb{R}^d)$
	such that $f_\delta(\mathbf{x})=0$ if and only if $\|\mathbf{x}\|\leq r-2\delta$ and that $f_\delta(\mathbf{x})=1$  if $\| \mathbf{x}\| \geq r-\delta$.  Therefore, combining \eqref{appl-invariance-principle} and \eqref{e11}, for each fixed $s\in (0,1)$ and 
	$\delta\in (0, r/2)$,
	\begin{align}\label{e12}
		\limsup_{n\to\infty}  n^{\frac{1}{\alpha-1}} \mathbb{P}_{\delta_{\sqrt{n}\mathbf{y}}}\big(M_n^d > \sqrt{n}r\big) 
		& \le    -\frac{1}{(1-s)^{\frac{1}{\alpha-1}}}\log \mathbb{E}_{\delta_{(1-s)^{-1/2}\mathbf{y}}} \Big(e^{- \frac{c_*(1-s)^{\frac{1}{\alpha-1}}}{2}s^{-\frac{1}{\alpha-1}} \langle f_\delta, X_1\rangle}\Big) + C_* s\nonumber\\
		& \le -\frac{1}{(1-s)^{\frac{1}{\alpha-1}}}\log \mathbb{P}_{\delta_{(1-s)^{-1/2}\mathbf{y}}} \big( \langle f_\delta, X_1\rangle =0\big) + C_* s\nonumber\\
		& = -\frac{1}{(1-s)^{\frac{1}{\alpha-1}}}\log \mathbb{P}_{\delta_{(1-s)^{-1/2}\mathbf{y}}}\big(  X_1((\overline{D}_{r-2\delta})^c)= 0\big) + C_* s,
	\end{align}
	where in the last equality we used the fact that   the support of $f_\delta$ is equal to $D_{r-2\delta}^c$.  
	From \eqref{scaling-property} (with $t=1/\kappa$ and $\kappa =r^2/ (r-2\delta)^2$ ), it holds that 
	\begin{align}
		-\log \mathbb{P}_{\delta_{(1-s)^{-1/2}\mathbf{y}}} \big(  X_1((\overline{D}_{r-2\delta})^c)= 0\big)  = - \kappa^{-\frac{1}{\alpha-1}} \log 	\mathbb{P}_{\delta_{(1-s)^{-1/2}\mathbf{y}/\sqrt{\kappa}}} \big(  X_{1/\kappa}(\overline{D}_r)^c)= 0\big)  .
	\end{align}
	Set $\varphi(\mathbf{x}):= -\log \mathbb{P}_{\delta_{\mathbf{x}}} (X_{1/2} ((\overline{D}_r)^c)=0)$. Let $\delta$ be sufficient small such that $\kappa < 2$, then by the Markov property at $1/2$, we have
	\begin{align}
		& -\log \mathbb{P}_{\delta_{(1-s)^{-1/2}\mathbf{y}}} \big(  X_1((\overline{D}_{r-2\delta})^c)= 0\big)  \\
		&= -\kappa^{-\frac{1}{\alpha-1}}  \log 	\mathbb{E}_{\delta_{(1-s)^{-1/2}\mathbf{y}/\sqrt{\kappa}}} \Big( \mathbb{P}_{\delta_{(1-s)^{-1/2}\mathbf{y}/\sqrt{\kappa}}}\big( X_{1/\kappa}((\overline{D}_r)^c)= 0| X_{1/\kappa-1/2}\big)\Big)  \\
		& =  -\kappa^{-\frac{1}{\alpha-1}}  \log 	\mathbb{E}_{\delta_{(1-s)^{-1/2}\mathbf{y}/\sqrt{\kappa}}} \big( e^{-\langle \varphi, X_{1/\kappa -1/2} \rangle}\big) = \kappa^{-\frac{1}{\alpha-1}}  u_\varphi(1/\kappa-1/2, (1-s)^{-1/2}\mathbf{y}/\sqrt{\kappa}).
	\end{align}
	Since $\varphi$ is a bounded measurable function,  $u_\varphi(t,\mathbf{y})$ is continuous in $(t,\mathbf{y})$. Therefore, taking $s\to 0$ first and then $\delta\to 0$ (or $\kappa \to 1$) in \eqref{e12}, we finally conclude that 
	\begin{align}\label{e12'}
		\limsup_{n\to\infty}  n^{\frac{1}{\alpha-1}} \mathbb{P}_{\delta_{\sqrt{n}\mathbf{y}}}\big(M_n^d > \sqrt{n}r\big) 
		& \le    u_\varphi( 1/2,  \mathbf{y} ) = -\log \mathbb{P}_{\delta_{\mathbf{y}}} \big(  X_1((\overline{D}_{r})^c)= 0\big).
	\end{align}
	By \eqref{N-measures}, the last term is equal to $\mathbb{N}_\mathbf{y}(w_1((\overline{D}_r)^c)>0)= \mathbb{N}_\mathbf{y}(M_1^X> r)$. Plugging this back to \eqref{e12'} implies \eqref{Goal2}.
	Combining \eqref{Goal1} and \eqref{Goal2}, we complete the proof of the theorem. 
\end{proof}

\begin{proof}[Proof of Corollary \ref{Cor}]
	As summarized in \cite[(1.10) and (1.16) below]{HRS2026},  we have 
	\begin{align}
		\lim_{n\to \infty}n^{\frac{1}{\alpha-1}}\mathbb{P}_{\delta_{\sqrt{n}\mathbf{y}}}\big(Z_n(\mathbb{R}^d)>0\big)
		=\mathbb{N}_{\mathbf{y}}\big(w_1(\mathbb{R}^d)>0\big).
	\end{align}
	This combined with Theorem \ref{thm-behavior-Mn} yields that for any 
	$r>0$,
	\begin{align}
		&\lim_{n\to\infty}\mathbb{P}_{\delta_{\sqrt{n}\mathbf{y}}} \big(M_n^d > \sqrt{n}r|Z_n(\mathbb{R}^d)>0\big)
		=\lim_{n\to\infty}\frac{\mathbb{P}_{\delta_{\sqrt{n}\mathbf{y}}} \left(M_n^d > \sqrt{n}r\right)}{\mathbb{P}_{\delta_{\sqrt{n}\mathbf{y}}} \left(Z_n(\mathbb{R}^d)>0\right)}\\
		& =\frac{\mathbb{N}_\mathbf{y}\big( M_1^{X}> r\big)}{\mathbb{N}_\mathbf{y}\left(w_1(\mathbb{R}^d)>0\right)}
		=   \mathbb{N}_\mathbf{y}\big(M_1^{X}> r|w_1(\mathbb{R}^d)\neq 0\big).
	\end{align}
	The desired result follows immediately from the standard criterion for weak convergence of one-dimensional distributions.
\end{proof}

\subsection{Proof of Theorem \ref{thm-total-progeny}}\label{S4.4}

For simplicity, for any $x>0$, $\mathbf{y} \in \overline{D}_1$ and $\gamma \geq 0$, we define 
\begin{align}
	U_{(\gamma,\infty)}(\mathbf{y};x):= \lim_{\theta\to\infty} U_{(\gamma,\theta)}(\mathbf{y};x) = x^{\frac{2}{\alpha-1}} \bigg( 1-\mathbb{E}_{\delta_{x\mathbf{y}}}\Big(e^{- \gamma x^{-\frac{2\alpha}{\alpha-1}} \zeta_x} 1_{\{\langle 1, Z_{\overline{D}_x}\rangle=0\}} \Big)\bigg).
\end{align}
\begin{lemma}\label{lemma-asymptotic-Laplace-transfrom}
	For any $\mathbf{y} \in D_1 $ and $\gamma \ge 0$, we have
	\begin{align}
		\lim_{x\to\infty} U_{(\gamma,\infty)}(\mathbf{y};x) =  V_{(\gamma, \infty)}(\mathbf{y}).
	\end{align}
\end{lemma}
\begin{proof}
	Define
	\begin{align}\label{Goal5}
		A_{\mathbf{y}}(x,\gamma,\theta):=  \left|U_{(\gamma,\infty)}(\mathbf{y};x) - U_{(\gamma,\theta)}(\mathbf{y};x)\right|.
	\end{align}
	According to the definition of $U_{(\gamma, \infty)}$ and $U_{(\gamma, \theta)}$, we see that 
	\begin{align}
		&A_{\mathbf{y}}(x,\gamma,\theta)
		=x^{\frac{2}{\alpha-1}} \Big|\mathbb{E}_{\delta_{x\mathbf{y}}}\Big(e^{- \gamma x^{-\frac{2\alpha}{\alpha-1}} \zeta_x} 1_{\{\langle 1, Z_{\overline{D}_x}\rangle=0\}} \Big)
		-\mathbb{E}_{\delta_{x\mathbf{y}}}\Big(e^{- \gamma x^{-\frac{2\alpha}{\alpha-1}} \zeta_x- \theta x^{-\frac{2}{\alpha-1}}\langle 1, Z_{\overline{D}_x}\rangle} \Big) \Big| \\
		&=x^{\frac{2}{\alpha-1}} \mathbb{E}_{\delta_{x\mathbf{y}}}
		\bigg(e^{- \gamma x^{-\frac{2\alpha}{\alpha-1}} \zeta_x} \Big(e^{-\theta x^{-\frac{2}{\alpha-1}} \langle 1, Z_{\overline{D}_x}\rangle}-1_{\{\langle 1, Z_{\overline{D}_x}\rangle=0\}} \Big)\bigg)\\
		&\leq x^{\frac{2}{\alpha-1}} \mathbb{E}_{\delta_{x\mathbf{y}}}
		\Big(e^{-\theta x^{-\frac{2}{\alpha-1}} \langle 1, Z_{\overline{D}_x}\rangle}-1_{\{\langle 1, Z_{\overline{D}_x}\rangle=0\}} \Big) 
		= x^{\frac{2}{\alpha-1}}\mathbb{P}_{\delta_{x\mathbf{y}}}(M^d>x) - U_{(0,\theta)}(\mathbf{y};x).
	\end{align}
	Combining Theorem \ref{thm-behavior-all-time-maximum} and Proposition \ref{Uniqueness}, we obtain 
	\begin{align}
		& \limsup_{x\to\infty} A_{\mathbf{y}}(x,\gamma,\theta)
		\le \mathbb{N}_\mathbf{y}\big(M^{X,d}\geq 1\big)- V_{(0,\theta)}(\mathbf{y})= V_{(0,\infty)}(\mathbf{y})-V_{(0,\theta)}(\mathbf{y}) =:B_\mathbf{y}(\theta).
	\end{align}
	Therefore, for each fixed $\theta>0$, it follows from Proposition \ref{Uniqueness}  that 
	\begin{align}\label{two-side-bound-U-gamma-infty}
		V_{(\gamma, \theta)}(\mathbf{y}) - B_\mathbf{y}(\theta) \leq \liminf_{x\to\infty} U_{(\gamma,\infty)}(\mathbf{y};x)  \leq 
		\limsup_{x\to \infty} 
		U_{(\gamma,\infty)}(\mathbf{y};x) \leq V_{(\gamma, \theta)}(\mathbf{y}) + B_\mathbf{y}(\theta) .
	\end{align}
	Taking $\theta\to\infty$ in \eqref{two-side-bound-U-gamma-infty} and together with $\lim_{\theta\to\infty} B_\mathbf{y}(\theta)=0$ implies the desired result. 
\end{proof}

\begin{lemma}\label{lemma-aux-result}
	For any $\gamma\geq 0$ and $\mathbf{y} \in D_1$,
	\begin{align}\label{Goal8}
		V_{(\gamma, \infty)}(\mathbf{y})-\left(\gamma/{\mathcal{C}(\alpha)}\right)^{\frac{1}{\alpha}} = \mathbb{N}_\mathbf{y}\Big(e^{-\gamma \zeta^X}1_{\{M^{X,d}\geq 1\}}\Big).
	\end{align}
\end{lemma}
\begin{proof}
	The case $\gamma=0$ follows directly by Proposition \ref{Uniqueness}, so we only consider the case $\gamma>0$. We first show that for any $\gamma>0$ and $\mathbf{y}\in \mathbb{R}^d$, 
	\begin{align}\label{Goal7}
		-\log \mathbb{E}_{\delta_\mathbf{y}}\Big(\exp\Big\{-\gamma \int_0^\infty \langle 1, X_t\rangle \mathrm{d}t\Big\}\Big) = \Big(\frac{\gamma}{\mathcal{C}(\alpha)}\Big)^{\frac{1}{\alpha}}.
	\end{align}
	Noticing that the left-hand side of \eqref{Goal7} is independent of $\mathbf{y}$.  
	According to the Markov property, for each fixed $T>0$, we have
	\begin{align}
		w_\gamma&:=-\log \mathbb{E}_{\delta_\mathbf{y}}\Big(\exp\Big\{-\gamma \int_0^\infty \langle 1, X_t\rangle \mathrm{d}t\Big\}\Big)= -\log \mathbb{E}\Big(\exp\Big\{-\gamma \int_0^T \langle 1, X_t\rangle \mathrm{d}t - w_\gamma \langle 1, X_T\rangle\Big\}\Big) . 
	\end{align}
	From \cite[Corollary 5.17]{LZbook}  (with $f=w_\gamma,g=\gamma$ and $t=T$), $w_\gamma$ ($= v_T(\mathbf{x})$ independnt of $T$ and $\mathbf{x}$) is the unique locally bounded positive solution to
	\[
	w_\gamma + T \mathcal{C}(\alpha) w_\gamma^\alpha =w_\gamma +T\gamma,
	\]
	which is equivalent to 
	$w_\gamma = (\gamma/\mathcal{C}(\alpha))^{\alpha^{-1}}$
	and this completes the proof of \eqref{Goal7}. Now combining \eqref{Appl-N-measure} and \eqref{Goal7} implies 
	\begin{align}\label{Goal7'}
		\mathbb{N}_{\mathbf{y}}\Big(1-\exp\left\{-\gamma \zeta^X\right\}\Big) = \left(\gamma/\mathcal{C}(\alpha)\right)^{\frac{1}{\alpha}}.
	\end{align}
	Since $X_t^{D_1}= X_t$ for all $t>0$ on 
	$\{M^{X,d}<1\} = \{\langle 1, X^{\overline{D}_1} \rangle=0\}$, 
	it follows from 
	\eqref{expression-V-eta-theta}
	and  \eqref{Goal7'}  that 
	\begin{align}
		&V_{(\gamma, \infty)}(\mathbf{y})-\left(\gamma/{\mathcal{C}(\alpha)}\right)^{\frac{1}{\alpha}}  \\
		=&  	-\log \mathbb E_{\delta_{\mathbf y}}
		\Big[\exp\Big\{-\gamma \int_0^\infty \langle 1, X_t^{D_1}\rangle \mathrm{d}t \Big\}1_{\{ M^{X,d}<1\}}\Big]  -	\mathbb{N}_{\mathbf{y}}\Big(1-\exp\big\{-\gamma \zeta^X\big\}\Big) \\
		=& 	-\log \mathbb E_{\delta_{\mathbf y}}
		\Big[\exp\Big\{-\gamma \int_0^\infty \langle 1, X_t\rangle \mathrm{d}t \Big\}1_{\{ M^{X,d}<1\}}\Big]  -	\mathbb{N}_{\mathbf{y}}\Big(1-\exp\big\{-\gamma \zeta^X\big\}\Big).
	\end{align}
	Combining \eqref{Appl-N-measure}, by an argument similar to that leading to \eqref{e13}, with the above equation, we obtain
	\begin{align}
		V_{(\gamma, \infty)}(\mathbf{y})-\left(\gamma/{\mathcal{C}(\alpha)}\right)^{\frac{1}{\alpha}}  &=	 \mathbb N_{\mathbf y}
		\Big[1- \exp\left\{-\gamma  \zeta^X \right\}1_{\{ M^{X,d}<1\}}\Big]  -	\mathbb{N}_{\mathbf{y}}\Big(1-\exp\big\{-\gamma \zeta^X\big\}\Big) \\
		& = \mathbb{N}_\mathbf{y}\Big(e^{-\gamma \zeta^X}1_{\{M^{X,d}\geq 1\}}\Big),
	\end{align}
	which implies the desired result.
\end{proof}

\begin{proof}[Proof of Theorem \ref{thm-total-progeny}]
	First we show that 
	for any $x>0$ and $\mathbf{y} \in D_1$, 
	\begin{align}\label{Goal6}
		\lim_{\gamma\downarrow 0+}  \gamma^{-\frac{1}{\alpha}}
		\mathbb{E}_{\delta_{x \mathbf{y}}}
		\big(1-e^{-\gamma \zeta} \big)= \mathcal{C}(\alpha)^{-\frac{1}{\alpha}}.
	\end{align}
	By the branching property
	and noticing that 
	$\mathbb{E}_{\delta_{x\mathbf{y}}}\left(e^{-\gamma \zeta} \right)$ is independent of $x$ and $\mathbf{y}$, 
	we obtain 
	$$\mathbb{E}_{\delta_{x\mathbf{y}}}\big(e^{-\gamma \zeta} \big)=\mathbb{E}\big(e^{-\gamma \zeta} \big)= e^{-\gamma} \sum_{k=0}^\infty p_k\mathbb{E}\big(e^{-\gamma \zeta} \big)^k.$$
	Thus, $\widehat{u}_\gamma:= 1- \mathbb{E}\left(e^{-\gamma \zeta} \right)$ solves equation 
	\begin{align}
		\widehat{u}_\gamma & = 1-e^{-\gamma} -e^{-\gamma} \phi(\widehat{u}_\gamma) + e^{-\gamma} \widehat{u}_\gamma\quad \Longleftrightarrow\quad \widehat{u}_\gamma = 1- \frac{e^{-\gamma}}{1-e^{-\gamma}} \phi(\widehat{u}_\gamma).
	\end{align}
	Since $\lim_{\gamma\downarrow 0} \widehat{u}_\gamma=0$ due to the criticality of the branching process,
	by Lemma \ref{lemma-property-phi},
	\[
	\lim_{\gamma\downarrow 0} \frac{e^{-\gamma}}{1-e^{-\gamma}} \phi(\widehat{u}_\gamma)=1\quad \Longleftrightarrow \quad 	\lim_{\gamma\downarrow 0} \frac{\mathcal{C}(\alpha)(\widehat{u}_\gamma)^\alpha}{\gamma} =1,
	\]
	which implies \eqref{Goal6}.
	Noticing that 
	\begin{align} 
		&  \mathbb{E}_{\delta_{x\mathbf{y}}}\Big(e^{-\gamma x^{-\frac{2\alpha}{\alpha-1}}\zeta} \big| M^d \geq x\Big)
		=   \frac{\mathbb{E}_{\delta_{x\mathbf{y}}}\Big(e^{-\gamma x^{-\frac{2\alpha}{\alpha-1}}\zeta}1_{\{M^d >x\}} \Big)+\mathbb{E}_{\delta_{x\mathbf{y}}}\Big(e^{-\gamma x^{-\frac{2\alpha}{\alpha-1}}\zeta}1_{\{M^d =x\}} \Big) }{\mathbb{P}_{\delta_{x\mathbf{y}}}\left(M^d \geq x\right)}\\
		&= \frac{\mathbb{E}_{\delta_{x\mathbf{y}}}\Big(e^{-\gamma x^{-\frac{2\alpha}{\alpha-1}}\zeta} \Big)-\mathbb{E}_{\delta_{x\mathbf{y}}}\Big(e^{-\gamma x^{-\frac{2\alpha}{\alpha-1}}\zeta_x} 1_{\{\langle 1, Z_{\overline{D}_x}\rangle=0\}} \Big) +\mathbb{E}_{\delta_{x\mathbf{y}}}\Big(e^{-\gamma x^{-\frac{2\alpha}{\alpha-1}}\zeta}1_{\{M^d =x\}} \Big)}{\mathbb{P}_{\delta_{x\mathbf{y}}}\left(M^d \geq x\right)}.
	\end{align}
	From Theorem \ref{thm-behavior-all-time-maximum}, we see that 
	\[
	\lim_{x\to\infty} x^{\frac{2}{\alpha-1}}\mathbb{E}_{\delta_{x\mathbf{y}}}\Big(e^{-\gamma x^{-\frac{2\alpha}{\alpha-1}}\zeta}1_{\{M^d =x\}} \Big)=0,\quad \lim_{x\to\infty}x^{\frac{2}{\alpha-1}}\mathbb{P}_{\delta_{x\mathbf{y}}}\big(M^d\geq x\big) = \mathbb{N}_\mathbf{y}\big(M^{X,d}\geq 1\big).
	\]
	Therefore, combining  Proposition \ref{Uniqueness}, Lemma \ref{lemma-asymptotic-Laplace-transfrom}, \eqref{Goal6} and the above 
	two displays,
	\begin{align}\label{equivalent-eq-conditional-expectation}
		&\qquad \lim_{x\to\infty} \mathbb{E}_{\delta_{x\mathbf{y}}}\Big(e^{-\gamma x^{-\frac{2\alpha}{\alpha-1}}\zeta} \big| M^d \geq x\Big)= \frac{V_{(\gamma, \infty)}(\mathbf{y})-\left(\gamma/{\mathcal{C}(\alpha)}\right)^{\frac{1}{\alpha}}}{ \mathbb{N}_\mathbf{y}\left(M^{X,d}\geq 1\right)}.
	\end{align}
	Now we complete the proof of the theorem by Lemma \ref{lemma-aux-result}. 
\end{proof}

\bigskip
	\vspace{.1in}

\end{document}